\documentclass[12pt]{article}

\usepackage[a4paper,left=17mm,top=20mm,right=17mm,bottom=25mm]{geometry}

\usepackage{amsmath, amsfonts, amssymb, amsthm}
\usepackage{mathtools}
\usepackage[mathscr]{euscript}

\usepackage{graphicx}
\usepackage{enumitem}
\usepackage{float}

\usepackage{tensor}

\usepackage{authblk}
\usepackage[numbers, sort&compress]{natbib}

\usepackage{stmaryrd}

\usepackage[colorlinks]{hyperref}
\hypersetup{
  citecolor=blue,
  linkcolor=blue,
}

\usepackage[font=footnotesize,width=.85\textwidth,labelfont=bf]{caption}

\newtheorem{theorem}{Theorem}[section]
\newtheorem{lemma}[theorem]{Lemma}
\newtheorem{corollary}[theorem]{Corollary}
\newtheorem{definition}[theorem]{Definition}
\newtheorem{proposition}[theorem]{Proposition}
\newtheorem{example}[theorem]{Example}
\newtheorem{problem}[theorem]{Problem}
\newtheorem{remark}[theorem]{Remark}

\usepackage[dvipsnames]{xcolor}
\definecolor{gray}{rgb}{0.67, 0.67, 0.67}
\definecolor{bgreen}{rgb}{0.2, 0.7, 0.2}

\DeclareMathOperator{\lprod}{\circ}
\DeclareMathOperator{\cprod}{\Box}
\DeclareMathOperator{\dprod}{\times}
\DeclareMathOperator{\sprod}{\boxtimes}

\DeclareFontFamily{OT1}{pzc}{}
\DeclareFontShape{OT1}{pzc}{m}{it}{<-> s * [1.100] pzcmi7t}{}
\DeclareMathAlphabet{\mathpzc}{OT1}{pzc}{m}{it}

\newcommand{\matchno}{\mathpzc{m}}
\newcommand{\unmatchno}{\mathpzc{u}}

\newcommand{\MB}{M^{\Box}}
\newcommand{\MU}{M^{\mathpzc{U}}}

\newcommand{\lefts}[2]{\prescript{#1}{}{\hspace{-0.02in}#2}}
\newcommand{\leftsprime}[2]{\prescript{#1}{}{\hspace{-0.04in}#2}}

\newcommand{\fM}{\mathpzc{M}^{\mathpzc{x}}} 

\newcommand{\fbast}{\mathpzc{M}^{\boxast}} 
\newcommand{\fast}{\mathpzc{M}^{\ast}} 
\newcommand{\fcast}{\mathpzc{M}^{\circledast}} 

    {\begin{description}[leftmargin = 0.2cm, labelsep = 0.2cm]}
    {\end{description}}

\newcommand{\PROBLEM}[1]{{\sc #1}}

\providecommand{\keywords}[1]{\textbf{\textit{Keywords: }} #1}

\title{\sc Construction of $k$-matchings and $k$-regular subgraphs in graph products\thanks{Dedicated to Wilfried Imrich on occasion of his 80th birthday.
   We have to blame Wilfried. If Wilfried would not have infected MH
   with his enthusiasm about graph products and thus, by way of transitivity,
   also AL, this paper would never have been written.}}

\author[ ]{Anna Lindeberg}
\author[ ]{Marc Hellmuth}

\affil[ ]{Department of Mathematics, Faculty of Science, Stockholm University,
  SE - 106 91 Stockholm,   Sweden } 

\date{\ }

\begin{document}

\maketitle

\abstract{
    A $k$-matching $M$ of a graph $G=(V,E)$ is a subset $M\subseteq E$ such that each
    connected component in the subgraph $F = (V,M)$ of $G$ is either a single-vertex
    graph or $k$-regular, i.e., each vertex has degree $k$. In this contribution, we
    are interested in $k$-matchings within the four standard graph products: the
    Cartesian, strong, direct and lexicographic product.

	  As we shall see, the problem of finding non-empty $k$-matchings ($k\geq 3$) in
	  graph products is NP-complete. Due to the general intractability of this problem,
	  we focus on distinct polynomial-time constructions of $k$-matchings in a graph
	  product $G\star H$ that are based on $k_G$-matchings $M_G$ and $k_H$-matchings
	  $M_H$ of its factors $G$ and $H$, respectively. In particular, we are interested
	  in properties of the factors that have to be satisfied such that these
	  constructions yield a maximum $k$-matching in the respective products. Such
	  constructions are also called ``well-behaved'' and we provide several
	  characterizations for this type of $k$-matchings.

		Our specific constructions of  $k$-matchings in graph products satisfy the property of being
		weak-homomorphism preserving, i.e., constructed matched edges in the product
		are never ``projected'' to unmatched edges in the factors.
		This leads to the concept of weak-homomorphism preserving $k$-matchings.
		Although the specific $k$-matchings constructed here are not always maximum $k$-matchings of the products,
	  they have always maximum size among all weak-homomorphism preserving $k$-matchings.
	  Not all weak-homomorphism preserving $k$-matchings, however, can be constructed in our manner.
		We will, therefore, determine the size of maximum-sized elements
		among all weak-homomorphims preserving $k$-matching within the respective graph products,
		provided that the matchings in the factors satisfy some general assumptions.
}

\newcommand{\sep}{; }
\bigskip
\noindent

\sloppy

\section{Introduction}

Graphs and in particular graph products arise in a variety of different contexts,
from computer science \cite{AMA-07, JHH+10} to theoretical biology
\cite{SS04,Wagner:03a}, computational engineering \cite{KK-08, KR-04, HIKS-09,Hel-10,
HIK-13, HIK-13b} or just as natural structures in discrete mathematics
\cite{HammackEtAl2011}. In this contribution, we are interested in the structure of
maximum $k$-matchings in graph products that are based on $k'$-matchings in the
factors. In particular, we focus on the four standard products: the Cartesian product
$\cprod$, the strong product $\sprod$, the direct product $\dprod$ and the
lexicographic product $\lprod$, see \cite{HammackEtAl2011} for an excellent overview
on graph products.

In our notation, a $k$-matching $M$ is a subset $M \subseteq E$ of the edges of a
graph $G=(V,E)$ such that every vertex in $V$ is adjacent to either $0$ or $k$ edges
in $M$ for some positive integer $k$, see \cite{Plummer:07,LovaszPlummer1986} for
surveys. In other words, $M \subseteq E$ is a $k$-matching, if the connected
components of the subgraph $F=(V,M)$ of $G$ are single-vertex graphs or $k$-regular.
Hence, finding non-empty $k$-matchings is equivalent to the problem of finding
$k$-regular subgraphs. We note that our notation of $k$-matchings is also equivalent
to the notion of so-called \emph{component factors} \cite{Plummer:07}, i.e., sets of
edges $M\subseteq E$ such that each connected component in $(V,M)$ belongs to a set
of specified graph classes. In our setting, the graph classes are the single vertex
graphs and the $k$-regular graphs.

A $k$-matching $M$ is perfect, if each vertex is incident to some edge of $M$, and
thus, $F$ is a $k$-regular spanning graph of $G$. Determining whether a perfect
$k$-matching exists and, in the affirmative case, computing a perfect $k$-matching
can be done in polynomial-time \cite{MeijerEtAl2009,Anstee:85,CC:90}. In contrast,
the problem of determining whether a given graph has a non-empty $k$-matching, is
NP-complete for all integer $k\geq 3$ \cite{CC:90,Plesnik:84,STEWART:96}. Perfect
$k$-matchings are also known as $k$-factors
\cite{MeijerEtAl2009,Sumner1974,KirkpatrickHell1983,Tutte1947,Tutte1952,Kotzig1979,Plummer:07,Saito:91,ZAKS:71};
a term that we do not use in order to avoid confusion with the factors $G$ and $H$ of
a product $G\star H$. The $k$-matching number $\matchno_k(G)$ is the number of edges
contained in a maximum-sized $k$-matching of $G$.

$1$-matchings for graph products have been extensively studied, see e.g.\
\cite{Kotzig1979,HW:02,mahamoodian:1981,YZ:06,DZH:18,WMCM:19,bai2010lexicographic,AlmeidaEtAl2013,pisanski19831,AAL:07,LZL:12,Gauci2020PerfectMA}.
Most of these results are concerned with quite restricted subclasses of the Cartesian
product; restrictions that are needed to answer often quite difficult questions as
e.g.\ finding so-called 1-factorizations, determining so-called preclusion numbers or
enumeration of perfect $k$-matchings. $3$-matchings in the Cartesian product of
cycles have been studied in \cite{BS:19} and $k$-matchings of hypercubes (the
Cartesian product of edges) in \cite{Ramras:99}. In particular, it seems to be a
non-trivial endeavour to characterize $k$-matchings of general graph products with
arbitrary factors in terms that are solely based on the structure of these factors.
As shown in Fig.\ \ref{fig:inducedMatchings}, the graph $G\cprod H$ has a perfect
$1$-matching although none of its factors has one. A further example is provided in
Fig.\ \ref{fig:strange4M}, where the product $G\cprod K_2$ has a perfect
$4$-matching. However, the $4$-matching in Fig.\ \ref{fig:strange4M} (right) does not
seem to be related to the perfect $1$-matching of $K_2$ or the perfect $1$-, $2$- or
$3$-matching of $G$ in an obvious way. Note, neither of the graphs $G$ and $K_2$ has
a non-empty $4$-matching. In contrast, the converse is relatively simple to solve: if
$G$ has a non-empty (or perfect) $k$-matching $M_G$, then $G\cprod H$ has a non-empty
(or perfect) $k$-matching for all graphs $H$ which is obtained by ``distributing''
the edges $M_G$ along the copies of $G$. Such constructions are widely used as a
vehicle to establish many of the aforecited results.

\begin{figure}[t]
	\begin{center}
		\includegraphics[width=0.85\textwidth]{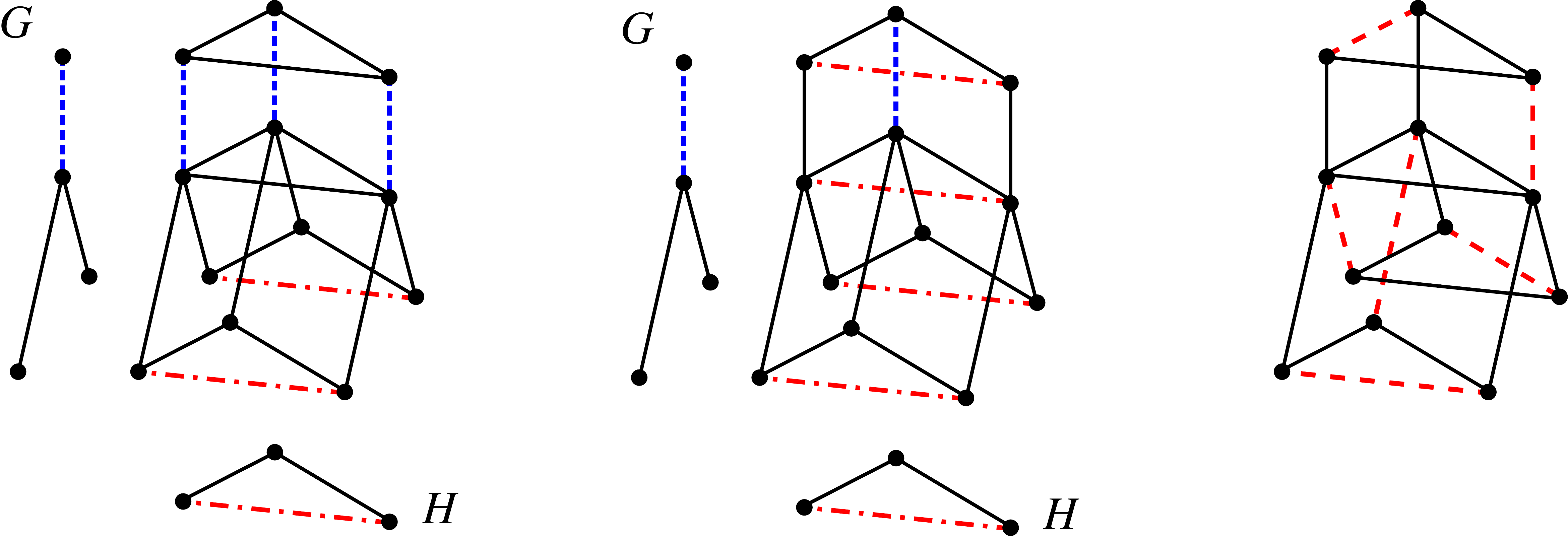}
	\end{center}
	\caption{Shown is the Cartesian product $G\Box H$. None of the factors has a
	         perfect $1$-matching although $G\Box H$ has one (highlighted by the dashed
	         edges in the right panel). The $1$-matchings $\fbast(M_{G}, M_{H})$ and
	         $\fbast(M_{H}, M_{G})$ provided by the $1$-matchings $M_{G}$ of $G$ and
	         $M_{H}$ of $H$ are shown in the left and middle panel, respectively.
	         Projections of the dashed edges $G\Box H$ onto the factors are either
	         single vertices or dashed edges in factors, i.e., weak homomorphisms
	         between the matched edges.}
		\label{fig:inducedMatchings}
\end{figure}

\begin{figure}[t]
	\begin{center}
		\includegraphics[width=0.85\textwidth]{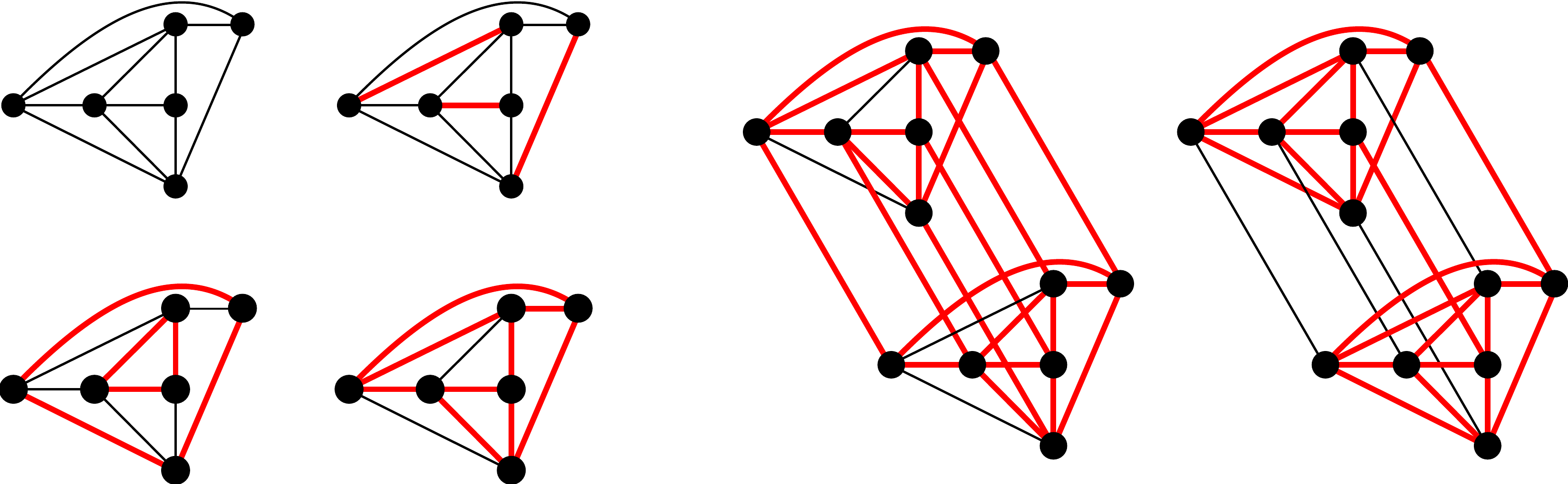}
	\end{center}
	\caption{\emph{Left Panel:} A graph $G$ together with a perfect $1$-matching, a perfect 2-matching and a
						perfect $3$-matching; highlighted by red-bold edges. Note, $G$ does not have a non-empty $4$-matching.
						\emph{Right panel:} The Cartesian product $G\cprod K_2$ together
					  with a two different perfect $4$-matchings; highlighted by red-bold edges.
					  The left $4$-matching consists of the copies of the $1$-matching of $K_2$ and the shown $3$-matching of $G$. }
		\label{fig:strange4M}
\end{figure}

While it seems to be straightforward to compose a $k$-matching in a Cartesian product
from the $k$-matchings of the factors, there are several ways to distribute a
$k$-matching of the factors along other types of products that additionally contain
``non-Cartesian'' edges. As we shall see in Section \ref{sec:construction},
finding $k$-matchings in graph products remains an NP-hard problem. Hence,
we will consider several constructions of $k$-matchings
$\fM(M_G,M_H)$ in a product $G\star H$ with $\star\in \{\cprod, \dprod, \sprod,
\lprod\}$ that are based on $k_G$-matchings $M_G$ and $k_H$-matchings $M_H$ of the
factors $G$ and $H$, respectively. In Section \ref{sec:construction}, we will
provide three specific types for such constructions
$\fM(M_G,M_H)$ and are, in particular, interested in properties of the factors that
have to be satisfied such that these constructions yield a maximum $k$-matching in
the respective products, i.e., $\matchno_K(G\star H) = |\fM(M_G,M_H)|$. Such
matchings are called ``well-behaved'' and we investigate their structure
in detail in Section \ref{sec:wellB}.

Our constructions $\fM(M_G,M_H)$ of $k$-matchings in the products $G\star H$ follow
unambiguous predefined rules based on the $k_G$-matching $M_G$ of $G$ and the
$k_H$-matching $M_H$ of $H$. To make this problem tractable and to avoid
``uncontrolled behaviour'' of the constructed $k$-matchings in the products, we have
chosen these rules in such way that they preserve weak-homomorphism between the
matched edges, i.e., matched edges in a product are never mapped onto unmatched edges
in the factors using the usual projection map \cite{HammackEtAl2011}. This leads to
the concept of weak-homomorphism preserving $k$-matchings, which is discussed in
Section \ref{sec:wh}. The set $\mathcal{W}_k(G\star H, M_G,M_H)$ denotes the
collection of \emph{all} $k$-matchings in $G\star H$ that are weak-homomorphism
preserving w.r.t.\ the matched edges in $M_G$ and $M_H$. As it turns out, the
three specific $k$-matchings constructed here have always maximum size among all
weak-homomorphism preserving $k$-matchings. Provided that the matchings $M_G$ and
$M_H$ satisfy some mild assumptions, we will furthermore determine the size of
maximum-sized elements within the set $\mathcal{W}_k(G\star H, M_G,M_H)$ for all
positive integers $k$. We finally collect some results for $1$-matchings based
on the established results in Section \ref{sec:1m} and
provide a short summary and discuss open problems in Section \ref{sec:sum}.

\section{Preliminaries}
\label{sec:prelim}

\paragraph{Basics.}
In this contribution, we consider finite, simple and undirected graphs $G=(V,E)$ with
non-empty vertex set $V_G\coloneqq V$ and edge set $E_G\coloneqq E\subseteq
\binom{V}{2}$. Note that, by definition, $G$ is loop-free, i.e., it does not
contain edges of the form $\{x,x\}$.

The \emph{order of $G$} is $|V_G|$ and its \emph{size} is $|E_G|$. The
\emph{degree $\deg_G(v)$} of a vertex $v\in V_G$ is the number of edges that are
incident to $v$. A graph is \emph{$k$-regular}, if every vertex of $G$ has degree $k$.
A graph $H$ is a \emph{subgraph} of a graph $G$, if $V_H\subseteq
V_G$ and $E_H\subseteq E_G$. A subgraph $H \subseteq G$ is \emph{induced} if for
every edge $\{x,y\}\in E_G$ with $x,y\in V_H$ we have $\{x,y\}\in E_H$.

Subsets $F,F'\subseteq E_G$ are called \emph{vertex disjoint} if 
for all $f\in F$ and $f'\in F'$ it holds that $f\cap f'=\emptyset$.
 
A path of length $n$ is denoted by $P_n$. The graph with vertex set $V$ and edge
set $\binom{V}{2}$ is a \emph{complete} graph and denoted by $K_{|V|}$. A \emph{star
$S_n$} is a tree containing $n+1$ vertices and one vertex of degree $n$.

Two graphs $G$ and $H$ are \emph{isomorphic}, in symbols $G\simeq H$, if there is a
bijective map $\varphi\colon V_G \to V_H$ such that $\{u,v\} \in  E_G$ 
if and only if $\{ \varphi(u), \varphi(v) \} \in E_H$.
For two graphs $G$ and $H$,
a \emph{weak homomorphism} is a map $\varphi\colon V_G \to V_H$
for which $\{u,v\} \in  E_G$ implies that
$\{ \varphi(u), \varphi(v) \} \in E_H$ or $\varphi(u) = \varphi(v)$.
In simple words, weak homomorphisms map edges to edges or to vertices.

\paragraph{Matchings and regular subgraphs.}

If $F=(V,M)$ is a subgraph of $G=(V,E)$ such that every vertex $v\in V$
has degree $0$ or $k$ in $F$, where $k$ is a positive integer, then $M$ is a \emph{$k$-matching of $G$}.
Hence, for a $k$-matching $M$ of $G$, the connected components in $F=(V,M)$ are $K_1$s or
$k$-regular graphs and if all isolated  vertices have been removed from $F$ we obtain a
$k$-regular or empty subgraph of $G$.

For $M\subseteq E$, we say that a
vertex $v\in V$ is \emph{$M$-unmatched} or \emph{unmatched by $M$}, if $v$ is not
incident to an edge in $M$. If there is
a $k$-matching $M$ of $G$ such that no vertex is $M$-unmatched, then $M$ is called a
\emph{perfect} $k$-matching of $G$ and, if only one vertex is $M$-unmatched,
a \emph{near-perfect} $k$-matching of $G$.

Note that $M$ may very well be empty and thus, for every positive integer $k$,
a $k$-matching of $G$ exists, namely $M=\emptyset$.
Furthermore, a
$k$-matching $M$ of $G$ is \emph{maximal} if there is no $k$-matching $M'$ of $G$ that
satisfies $M\subsetneq M'$. A $k$-matching $M$ of $G$ is \emph{maximum} if it has the
largest size among all $k$-matchings of $G$. The size of a maximum $k$-matching of
$G$ is denoted by $\matchno_k(G)$ and the number of vertices that are unmatched by a
maximum $k$-matching by $\unmatchno_k(G)$.

We emphasize that there are many ways to define particular types of matchings, cf.\ \cite{Plummer:07,LovaszPlummer1986}.
In this context, our notion of $k$-matchings is equivalent to so-called
\emph{component factors}  while
perfect $k$-matchings coincide with the well-known notion of \emph{$k$-factors} which are $k$-regular spanning graphs.
While it is an easy task to verify the existence and to find a perfect $k$-matching (if one exists) \cite{MeijerEtAl2009,Anstee:85,CC:90},
it is NP-complete to determine whether there is a non-empty $k$-matchings in a graph for all $k\geq 3$ \cite{CC:90,Plesnik:84,STEWART:96}.

We provide here a simple result for later reference.

\begin{lemma}
\label{lem:muFormula}
For all positive integers $k$ and every $k$-matching $M$ of a graph $G$ of order $n$
it holds that $|M| =k\cdot \frac{n-u_G}{2}$, where $u_G$ denotes the number of
$M$-unmatched vertices in $G$. In particular, $\matchno_k(G)=k\cdot
\frac{n-\unmatchno_k(G)}{2}$. Moreover, $G$ has a perfect $k$-matching if and only if
$\matchno_k(G)=k\cdot \frac{n}{2}$.
\end{lemma}
\begin{proof}
Let $M$ be a $k$-matching of $G=(V,E)$ and let $H=(V,M)$. Note, there are $u_G$
vertices $v\in V$ of degree $\deg_H(v)=0$ and $n-u_G$ vertices of degree
$\deg_H(v)=k$. Hence, $\sum_{v\in V}\deg_H(v) = (n-u_G)k = 2|M|$ which is, if and
only if, $|M| =k\cdot \frac{n-u_G}{2}$. If $M$ is a maximum $k$-matching, then
$|M|=\matchno_k(G)$ and $\unmatchno_k(G) = u_G$ and so, $\matchno_k(G)=k\cdot
\frac{n-\unmatchno_k(G)}{2}$. Finally, $G$ has a perfect $k$-matching if and only if
$\unmatchno_k(G) = 0$, which is, by the latter arguments, if and only if
$\matchno_k(G)=k\cdot \frac{n}{2}$.
\end{proof}

\paragraph{Graph Products.}
Given graphs $G$ and $H$, there are four standard graph products: the Cartesian
product $G\cprod H$, the direct product $G\dprod H$, the strong product $G\sprod H$
and the lexicographic product $G\lprod H$ \cite{HammackEtAl2011}. The vertex set
$V_{G\star H}$ of each of the products $\star\in \{\cprod, \sprod,\lprod, \dprod \}$ is
defined as the Cartesian set product $V_G\times V_H$ of the vertex sets of the
factors.
Two vertices $(g,h), (g',h') \in V_G\times V_H$ are adjacent in $G\sprod H$
precisely if one of the following conditions is satisfied:
\begin{itemize}[noitemsep]
	\item[(i)]  $\{g,g'\}\in E_G$ and $h=h'$.
	\item[(ii)] $g=g'$ and $\{h,h'\}\in E_H$.
	\item[(iii)] $\{g,g'\}\in E_G$ and $\{h,h'\}\in E_H$.
\end{itemize}
In the lexicographic product $G\lprod H$ there is an edge $\{(g,h),(g',h')\}$ if and
only if $\{g,g' \} \in E_G$ or (ii) is satisfied. In the Cartesian product $G \cprod
H$ vertices are only adjacent if they satisfy (i) or (ii). Consequently, $G\cprod H
\subseteq G\sprod H \subseteq G\lprod H$ and we call the edges of a strong and
lexicographic product that satisfy (i) or (ii) \emph{Cartesian} edges and the others
\emph{non-Cartesian} edges. In the direct product $G \dprod H$ vertices are adjacent
precisely if they satisfy (iii). Hence, $G\dprod H \subseteq G\sprod H \subseteq G\lprod H$.

For $\star\in \{\cprod, \sprod,\lprod \}$ and a given a vertex $v = (g, h) \in
V_{G\star H}$, the graph $G^h$ and $\lefts{g}{H}$ denotes the subgraph of $G \star
H$ that is induced by the vertex set $\{(x,h)\colon x\in V_G\}$ and $\{(g,y)\colon y\in
V_H\}$, respectively. The subgraph $G^h$ is called a \emph{$G$-layer}
and $\lefts{g}{H}$ an \emph{$H$-layer}.
Note, that $G$- and $H$-layers are subgraphs of $G\cprod H$ that are isomorphic to $G$ and $H$,
respectively. Hence, we will refer to the subgraphs $G^h$ and $\lefts{g}{H}$
also as \emph{copies of} $G$ and $H$, respectively. Note that copies $G^h$
and $\lefts{g}{H}$ intersect in precisely one vertex, namely in $(g,h)$.
Moreover,
$V_{G^h}\cap V_{G^{h'}}=\emptyset$ if and only if $h\neq h'$, as well as
$V_{\lefts{g}{H}}\cap V_{\leftsprime{g'}{H}}=\emptyset$ if and only if $g\neq g'$.

Given a product $G\star H$, $\star\in \{\cprod, \dprod,\sprod,\lprod\}$,
the  \emph{projection  maps} $p_G\colon V_{G\star H} \to V_G$ and
$p_H\colon V_{G\star H} \to V_H$ are defined as
$p_G((g,h)) = g$ and $p_H((g,h)) = h$.
For $\star\in \{\cprod, \dprod,\sprod\}$, both  $p_G$ and $p_H$
are weak homomorphisms. For $G\lprod H$ only $p_G$ is a weak homomorphism.


\section{Construction of  $\boldsymbol{k}$-matching in products}
\label{sec:construction}

In this section, we are concerned with the constructions of $k$-matching in graph
products $G\star H$ where $\star\in \{\cprod, \sprod, \dprod, \lprod\}$, provided that
some information about a $k_G$-matching of $G$ and a $k_H$-matching of $H$ is
available. Determining whether $G$ has a non-empty $k$-matching is an NP-complete
problem for all $k\geq 3$ \cite{CC:90,Plesnik:84}. It is not difficult to prove that this
problem remains NP-complete if the input are graph products, as shown below.
Due to the intractability of the latter problem,  we
will consider specific constructions for $k$-matchings in graph products. We will
first characterize when such $k$-matchings in $G\star H$ exists. In Section
\ref{sec:wellB} we will discuss under which conditions such constructions yield
maximum $k$-matchings in $G\star H$. Our constructions are not chosen arbitrarily
and, in particular, are designed in a way that constructed matched edges in the
products are never projected onto unmatched edges in the factors, a property that is
reasonable to avoid to overly complicated or random constructions. We will discuss
this property in some further detail in Section \ref{sec:wh}.

Let us first briefly discuss
the computational complexity of finding $k$-matchings in graph products.
To this end, consider the following well-known problem which was initially
stated as ``$k$-regular subgraph problem'' \cite{CC:90,Plesnik:84}.

\begin{problem}[\PROBLEM{$k$-Matching}]\ \\
  \begin{tabular}{ll}
    \emph{Input:}    & An integer $k$ and a graph $G$.\\
    \emph{Question:} & Is there a non-empty $k$-matching of $G$?
  \end{tabular}
\end{problem}

Plesnik  \cite{Plesnik:84} provided
\begin{theorem}
For every $k\geq 3$,  \PROBLEM{$k$-Matching} is NP-complete, even for the class
of bipartite graphs  with maximum degree $k+1$.
\end{theorem}

Based on this result, we can easily show that the problem of determining whether a graph
product has a non-empty $k$-matching is NP-complete as well.

\begin{corollary}
For every $k\geq 3$, \PROBLEM{$k$-Matching} remains NP-complete when the input is restricted to
graph products $G\star H$,
$\star\in \{\cprod,\sprod,\lprod, \dprod\}$.
\end{corollary}
\begin{proof}
It suffices to show NP-hardness. We shall show that \PROBLEM{$k$-Matching} can be reduced to this problem.
To this end, let $k\geq 3$ be some integer and $G$ be an arbitrary bipartite graph  (with maximum degree $k+1$).

First consider $\star\in \{\cprod,\sprod,\lprod\}$. Construct an instance $G\star K_1$.
Since $G\star K_1\simeq G$, it trivially holds that $G$ has a non-empty $k$-matching if and only if
$G\star K_1$ has one.

Now consider the direct product $\star = \dprod$ and construct an instance $G\dprod K_2$.
By \cite[Cor.\ 1.2]{sampathkumar_1975}, the product $G\dprod K_2$ consists of two vertex disjoint copies
of $G$, since $G$ is bipartite. Hence, $G$ has a non-empty $k$-matching if and only if
$G\dprod K_2$ has a non-empty $k$-matching.
\end{proof}

As the problem of finding non-empty $k$-matchings in graph products is intractable in general,
we consider now the problem of constructing $k$-matchings in products provided that some information
about matchings in the respective factors is available.


\subsection{$\fbast$-constructions}

There are several ways to compose a $k$-matching in a
product for given $k$-matchings of the factors. We first consider the construction of
$k$-matchings in $G\star H$ by taking the copies of the matched edges of one factor
and then distribute the matched edges of the other factor along the unmatched
vertices. To be more precise, consider
\begin{definition}\label{def:bast}
	Let $G$ and $H$ be graphs and let $M_G\subseteq E_G$ and $M_H\subseteq E_H$.
	Furthermore, let $U_G\subseteq V_G$ and $U_H\subseteq V_H$ be the set of vertices
	that are unmatched by $M_G$ and $M_H$ in $G$ and $H$, respectively. For $G$ and
	$H$, we define the following sets
	\begin{align*}
		\MB_G \coloneqq & \left \{\{(g,h),(g',h)\} \mid \{g,g'\}\in M_G, h\in V_H \right \}  \\
		\MU_G \coloneqq & \left \{\{(u,h),(u,h')\} \mid u\in U_G, \{h,h'\}\in M_H  \right \} \\[0.1in]
		\MB_H \coloneqq & \left \{\{(g,h),(g,h')\} \mid \{h,h'\}\in M_H, g\in V_G \right \}  \\
		\MU_H \coloneqq & \left \{\{(g,u),(g',u)\} \mid u\in U_H, \{g,g'\}\in M_G  \right \}
   \end{align*}
	\[\fbast(M_G,M_H)   \coloneqq  \MB_G \cup \MU_G \text{ and }
				  \fbast(M_H,M_G)   \coloneqq  \MB_H \cup \MU_H.\]
\end{definition}
Note that, in general, $\fbast(M_G,M_H)\neq \fbast(M_H,M_G)$; see Fig.\ \ref{fig:inducedMatchings}
and \ref{fig:fbast-exmpl} for illustrative examples.
By construction, it always holds that $\fbast(M_G,M_H), \fbast(M_H,M_G)\subseteq
E_{G\star H}$, $\star\in \{\cprod,\sprod,\lprod\}$. In particular,
$\fbast(M_G,M_H)$ and $\fbast(M_H,M_G)$ consist of Cartesian edges only.

Note, we can always find subsets $M_G\subseteq E_G$ and $M_H\subseteq E_H$ such that
$\fbast(M_G,M_H)$ and $\fbast(M_H,M_G)$ are $k$-matchings, simply by putting
$M_G=M_H=\emptyset$. Moreover, the edges in non-empty $\fbast(M_G,M_H)$ do only
exist in the Cartesian, strong and lexicographic product but not in the direct product. Hence,
the products of interest in this part are the Cartesian, strong and lexicographic product.

\begin{figure}[t]
	\begin{center}
		\includegraphics[width=0.95\textwidth]{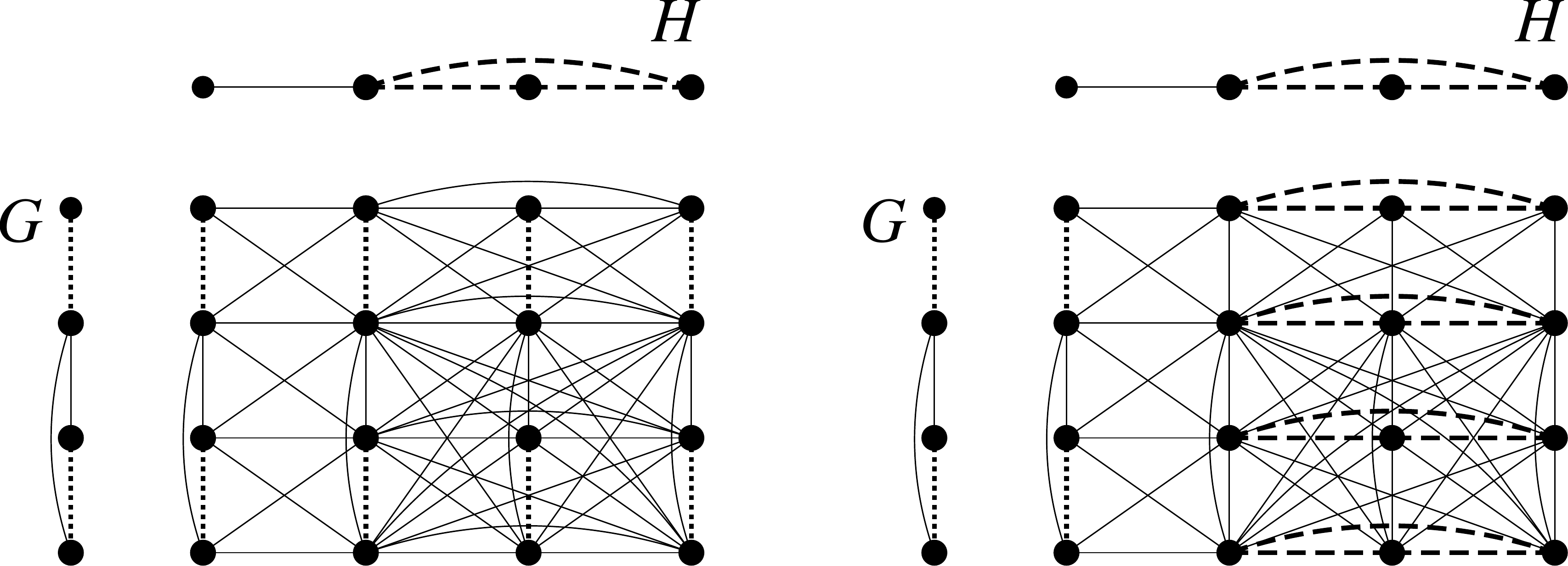}
	\end{center}
	\caption{Shown is the strong product $G\sprod H =(V,E)$ and $\fbast(M_G,M_H)$
	         (left) and $\fbast(M_H,M_G)$ (right) for the perfect $1$-matching $M_G$ of
	         $G$ and the maximal $2$-matching $M_H$ of $H$. \emph{Left:} The set
	         $\MB_G$ consists of the dotted-lined edges while $\MU_G=\emptyset$.
	         \emph{Right:} The set $\MB_H$ consists of the dashed-lined edges and
	         $\MU_H$ of the dotted-lined edges. Note, $\fbast(M_G,M_H)$ is a
	         $1$-matching of $G\sprod H$, while $\fbast(M_H,M_G)$ is not since the
	         subgraph $F=(V,\fbast(M_H,M_G))$ is not $k$-regular. }
	     		\label{fig:fbast-exmpl}
\end{figure}

We provide now a simple but technical result for later reference.
\begin{lemma}\label{lem:basics-MB-MG}
Let $\star\in \{\cprod,\sprod,\lprod\}$ and  $G\star H=(V,E)$ be a product with factors
$G$ and $H$. Then, $\MB_G$ and $\MU_G$ as well as $\MB_H$ and $\MU_H$
are vertex disjoint for all $M_G\subseteq E_G$ and $M_H\subseteq E_H$. 

Moreover, $(V_G\times \{h\},\MB_G\cap E_{G^h})\simeq (V_G,M_G)$ for all $h\in V_H$ and 
$(\{g\}\times V_H,\MU_G\cap E_{\lefts{g}{H}})\simeq (V_H,M_H)$ for all $g\in U_G$.
\end{lemma}
\begin{proof}
For the first statement, observe that $\{(g,h),(g',h)\} \in \MB_G$ implies that
$g,g'\notin U_G$ while $\{(g'',h'),(g'', h'')\} \in \MU_G$ implies that $g''\in U_G$.
Consequently, for edges $\{(g,h),(g',h)\} \in \MB_G$ and $\{(g'',h'),(g'',h'')\} \in \MU_G$ 
it must hold that $g,g'\neq g''$. Therefore, $\MB_G$ and $\MU_G$ are
vertex disjoint for all $M_G\subseteq E_G$ and $M_H\subseteq E_H$. By similar
arguments, $\MB_H$ and $\MU_H$ are vertex disjoint for all $M_G\subseteq E_G$ and
$M_H\subseteq E_H$.

For the second statement, consider first $F_{G^h} = (V_G\times \{h\},\MB_G\cap E_{G^h})$ with $h\in
V_H$. By construction, $\{(g,h),(g',h)\}\in \MB_G$ if and only
if $h\in V_H$ and $\{g,g'\}\in M_G$. Hence, $\{(g,h),(g',h)\}$ is an edge in $F_{G^h}$
if and only if and $\{g,g'\}\in M_G$. Thus, we readily obtain an isomorphism between
distinct $F_{G^{h}}$ and $F_{G^{h'}}$ as well as between $F_{G^h}$ and $F_G = (V_G,M_G)$ for all
$h,h'\in V_H$.

Now consider $F_{\lefts{g}{H}} = (\{g\}\times V_H,\MU_G\cap E_{\lefts{g}{H}})$ with $g\in U_G$ and
$F_H = (V_H,M_H)$. By construction, $\{(g,h),(g,h')\}\in \MU_G$ if and only if $g\in
U_G$ and $\{h,h'\}\in M_H$. Hence, $\{(g,h),(g,h')\}$ is an edge in $F_{\lefts{g}{H}}$
if and only if $g\in U_G$ and $\{h,h'\}\in M_H$. Again, we readily obtain an
isomorphism between distinct $F_{\lefts{g}{H}}$ and $F_{\leftsprime{{g'}}{H}}$ as
well $F_{\lefts{g}{H}}$ and $F_H$ for all $g,g'\in U_G$.
\end{proof}

We provide now a characterization for $k$-matchings $\fbast(M_G,M_H)$
		in terms of properties of $M_G$ and $M_H$.

\begin{proposition}\label{prop:matching-fac}
Let $\star\in \{\cprod,\sprod,\lprod\}$ and
$G$ and $H$ be graphs with $M_G\subseteq E_G$ and $M_H\subseteq E_H$.
Then, $\fbast(M_G,M_H)$ is a $k$-matching of
$G\star H$ if and only if one of the following statements
is satisfied.
\begin{enumerate}[itemsep=0.02em] 
	\item $M_G$ is a perfect $k$-matching while $M_H$ could be any set.
	\item $M_G$ and $M_H$ are $k$-matchings of $G$ and $H$, respectively.
\end{enumerate}
Analogous result hold for $\fbast(M_H, M_G)$.
\end{proposition}
\begin{proof} 
For the \emph{if}-direction,  assume first that $M_G$ is a perfect $k$-matching of $G$. In this case, Lemma
\ref{lem:basics-MB-MG} implies that $\MB_G\cap E_{G^h}$ is a (perfect) $k$-matching of
every copy $G^h$ of $G$ with $h\in V_H$. In particular, $U_G=\emptyset$ and hence,
$\MU_G=\emptyset$ and $\fbast(M_G,M_H) = \MB_G$. Thus, $\fbast(M_G,M_H)$ contains
only edges $\{(g,h),(g',h)\}$
with $\{g,g'\}\in M_G$ and $h\in V_H$. Since distinct
copies $G^h$ and $G^{h'}$ of $G$ are vertex disjoint and since $\MB_G\cap E_{G^h}$ is a
$k$-matching of every copy $G^h$ of $G$ in $G\star H$, it follows that
$\fbast(M_G,M_H)  = \MB_G$ is a $k$-matching of $G\star H$.

Now assume that $M_G$ and $M_H$ are $k$-matchings of $G$ and $H$, respectively. By
Lemma \ref{lem:basics-MB-MG}, $\MB_G\cap E_{G^h}$ is a $k$-matching of every copy
$G^h$ of $G$ and $\MU_G\cap E_{\lefts{g}{H}}$ is a $k$-matching of every copy
$\lefts{g}{H}$ of $H$ with $g\in U_G$ in $G\star H$. By Lemma \ref{lem:basics-MB-MG},
$\MB_G$ and $\MU_G$ are vertex disjoint. This together with the fact that no further
edges $\{(g,h),(g',h')\}$ with $g\neq g'$ and $h\neq h'$ are contained in
$\fbast(M_G,M_H)$ and that distinct copies $G^h$ and $G^{h'}$ of $G$ as well as distinct
copies $\lefts{g}{H}$ and $\lefts{g'}{H}$ of $H$ are vertex disjoint implies that
$\fbast(M_G,M_H)$ is a $k$-matching of $G\star H$.

For the \emph{only-if}-direction suppose that $\fbast(M_G,M_H)$ is a $k$-matching of
$G\star H$. By Lemma \ref{lem:basics-MB-MG}, $\MB_G$ and $\MU_G$ are vertex disjoint and
thus, $\MB_G$ and $\MU_G$ are $k$-matchings of $G\star H$.

By construction, either $\MB_G=\emptyset$ or $\MB_G$ contains solely edges of copies of $G$. This and the fact
that distinct copies $G^h$ and $G^{h'}$ of $G$ are vertex disjoint in $G\star H$
implies that $\MB_G\cap E_{G^h}$ is a $k$-matching in each copy $G^h$ of $G$.
Moreover, by Lemma \ref{lem:basics-MB-MG}, the graph $(V_G\times \{h\},\MB_G\cap E_{G^h})$ is
isomorphic to $(V_G,M_G)$ for every $h\in V_H$. Since $G$ and $G^h$ are isomorphic
for every $h\in V_H$, $M_G$ is a $k$-matching in $G$. Thus, if $U_G=\emptyset$ then,
$M_G$ is a perfect $k$-matching in $G$. Now assume that $U_G\neq \emptyset$. By
construction, every edge in $\MU_G$ must be located in some copy $\lefts{g}{H}$ of
$H$ in $G\star H$ where $g\in U_G$. Since distinct copies $\lefts{g}{H}$ and
$\lefts{g'}{H}$ of $H$ are vertex disjoint and since $\MU_G$ is a $k$-matching of
$G\star H$, we can conclude that $\MU_G\cap E_{\lefts{g}{H}}$ is a $k$-matching in
each copy $\lefts{g}{H}$ of $H$ with $g\in U_G$. Moreover, by Lemma
\ref{lem:basics-MB-MG}, the graph $(\{g\}\times V_H,\MU_G\cap E_{\lefts{g}{H}})$ is isomorphic to
$(V_H,M_H)$ for every $g\in U_G$. Consequently, $M_H$ is a $k$-matching in $H$.

By analogous arguments we obtain the proof for $\fbast(M_H, M_G)$.
\end{proof}

As a direct consequence of Prop.\ \ref{prop:matching-fac}, we obtain the following result.
\begin{corollary}
\label{cor:matching-fac}
Let $G$ and $H$ be graphs with $M_G\subseteq E_G$ and $M_H\subseteq E_H$ and
$\star\in \{\cprod,\sprod,\lprod\}$. Then, $\fbast(M_G,M_H)$ \emph{and}
$\fbast(M_H,M_G)$ are both $k$-matchings of $G\star H$ if and only if $M_G$ and $M_H$ are $k$-matchings
of $G$ and $H$, respectively.
\end{corollary}

For later reference, we provide the following simple result.

\begin{lemma}\label{lem:fbast-gh-unmatch}
Let $\star\in\{\cprod,\sprod,\lprod\}$, $G$ and $H$ be two graphs and
$M_G\subseteq E_G$ and $M_H\subseteq E_H$.
Then, for all $g\in V_G$ and $h\in V_H$,
the following statements are equivalent.
\begin{enumerate}
\item $g\in V_G$ is $M_G$-unmatched in $G$ and $h\in V_H$ is $M_H$-unmatched in $H$,
\item $(g,h)\in V_{G\star H}$ is $\fbast(M_G,M_H)$-unmatched in $G\star H$,
\item $(g,h)\in V_{G\star H}$ is $\fbast(M_H,M_G)$-unmatched in $G\star H$.
\end{enumerate} 
\end{lemma}
\begin{proof}
We will show (1) implies both (2) and (3), and that (2) implies (1) as well as (3) implies (1). 
To this end, first assume
(1) holds, so that $g$ is $M_G$-unmatched in $G$ and $h$ is $M_H$-unmatched in $H$. By definition there
are no edges in $\MB_G$ nor in $\MB_H$ incident to the vertex $v=(g,h)$. 
Moreover, if $\{(g,h),(g,h')\}\in\MU_G$, then $\{h,h'\}\in M_H$ per definition, 
contradicting $h$ being unmatched by $M_H$.
Hence, $\{(g,h),(g,h')\}\notin\MU_G$ and, by similar arguments, 
there is no incident edge of $v=(g,h)$ in $\MU_H$. 
Taken the latter arguments together, $v$ is both $\fbast(M_G,M_H)$-unmatched and 
$\fbast(M_H,M_G)$-unmatched in $G\star H$.

Now assume that (2) is satisfied and
let $(g,h)$ be an $\fbast(M_G,M_H)$-unmatched vertex in $G\star H$. If $g$ would
	be $M_G$-matched in $G$, then there is an edge $\{g,g'\}\in M_G$. By definition
	$\{(g,h),(g',h)\}\in \MB_G \subseteq \fbast(M_G,M_H)$ and thus, $(g,h)$ is not
	$\fbast(M_G,M_H)$-unmatched in $G\star H$; a contradiction. Assume, for contradiction,
	that $h$ is $M_H$-matched in $H$ and hence, there is an edge $\{h,h'\}\in M_H$. By
	the latter argument, $g$ must be $M_G$-unmatched in $G$ and thus, $g\in U_G$. By
	definition, $\{(g,h),(g,h')\}\in \MU_G\subseteq \fbast(M_G,M_H)$ and therefore,
	$(g,h)$ is not $\fbast(M_G,M_H)$-unmatched in $G\star H$; a contradiction.
	Therefore, $g$ is $M_G$-unmatched in $G$ and $h$ is $M_H$-unmatched in $H$
	and thus, Condition (1) is satisfied. By similar arguments one shows
	that (3) implies (1). 
\end{proof}


\subsection{$\fast$-constructions}

We continue with a natural way to construct $k$-matchings based on the non-Cartesian edges.

\begin{definition}\label{def:ast}
	Let $G$ and $H$ be graphs and let $M_G\subseteq E_G$ and $M_H\subseteq E_H$.
	Furthermore,	let $U_G\subseteq V_G$ and $U_H\subseteq V_H$ be the set
	of vertices that are unmatched by $M_G$ and $M_H$ in $G$ and $H$, respectively.
	For $G$ and $H$, we define the following set
   \begin{align*}
   \fast(M_G, M_H) 		\coloneqq &  \left \{\{(g,h),(g',h')\} \mid \{g,g'\}\in M_G, \{h,h'\}\in M_H \right \}.
   \end{align*}
\end{definition}

	We emphasize that, in contrast to $\fbast(M_G,M_H)$ and $\fbast(M_H,M_G)$,
	the sets    $\fast(M_G, M_H)$ and    $\fast(M_H, M_G)$ operate on different vertex sets,
	namely on $V_G\times V_H$ and $V_H\times V_G$, respectively.
Def.\ \ref{def:ast} is illustrated in Fig.\ \ref{fig:fcast-exmpl}.

Note, the edges in non-empty $\fast(M_G,M_H)$ do only exist in the strong, direct and
lexicographic product but not in the Cartesian product.
Hence, the products of interest in this part are the direct, strong
and lexicographic product. For later reference, we provide 
the following simple result.

\begin{lemma}\label{lem:degree-ast}
Let $\star\in \{\sprod,\lprod,\dprod\}$ and
$G$ and $H$ be graphs with $M_G\subseteq E_G$ and $M_H\subseteq E_H$.
Moreover, let $G\star H = (V,E)$ and $v=(g,h)\in V$ be a vertex that is incident to the edge $e \in E$.
If $e\in \fast(M_G,M_H)$, then $v$ has degree $k_g\cdot k_h$ in $F=\left(V,\fast(M_G,M_H)\right)$, 
where $k_g$ is the degree of $g$ in $(V_G,M_G)$ and $k_h$ is the degree of $h$ in $(V_H,M_H)$.

In particular, if $M_G$ is a $k_G$-matching of $G$ and $M_H$ a $k_H$-matching of $H$, 
then the degree of $v$ in $F$ is $k_G\cdot k_H$.
\end{lemma}

\begin{proof}
Let $v=(g,h)\in V$ be a vertex that is incident to the edge $e \in E$ such
that $e = \{(g,h),(g',h')\}\in \fast(M_G,M_H)$. Hence, by definition, $\{g,g'\}\in
M_G$ and $\{h,h'\}\in M_H$. This together with the fact that 
$g$ has degree $k_g$ in $(V_G,M_G)$ and $h$ has degree $k_h$ in $(V_H,M_H)$
 implies that $g$ must be
incident to precisely the $k_g$ edges $\{g,g_1\},\dots \{g,g_{k_g}\}\in M_G$ and $h$ to
precisely the $k_h$ edges $\{h,h_1\},\dots \{h,h_{k_h}\}\in M_H$. By definition, the
vertex $v=(g,h)$ is incident to the $k_g\cdot k_h$ edges $\{(g,h),(g_i,h_j)\} \in
\fast(M_G,M_H)$ with $1\leq i\leq k_g$ and $1\leq j\leq k_h$.

In particular, if $M_G$ is a $k_G$-matching of $G$ and $M_H$ a $k_H$-matching of $H$, 
then all vertices in $G$, resp., $H$ must have degree $0$ or $k_G$, resp., $k_H$. 
Since $v=(g,h)$ is incident to the edge $e \in \fast(M_G,M_H)$, the latter arguments
imply that  the degree of $v$ in $G\star H$ is $k_G\cdot k_H$.
\end{proof}

We provide now a characterization for $k$-matchings $\fast(M_G,M_H)$ in terms of properties of $M_G$ and $M_H$.

\begin{proposition}\label{prop:matching-ast}
Let $\star\in \{\sprod,\lprod,\dprod\}$ and $G$ and $H$ be graphs with $M_G\subseteq E_G$ and $M_H\subseteq E_H$.
Then, $\fast(M_G,M_H)$ is a non-empty $k$-matching of $G\star H$ if and only if
$M_G$ is a non-empty $k_G$-matching of $G$ and $M_H$ is a non-empty $k_H$-matching of $H$, where $k_G\cdot k_H=k$
\end{proposition}
\begin{proof}
	Let $G\star H = (V,E)$ and $F = (V, \fast(M_G,M_H))$. Suppose first that $M_G$ is a
	non-empty $k_G$-matching of $G$ and $M_H$ is a non-empty $k_H$-matching 
	of $H$, where $k_G, k_H\geq 1$ and $k_G\cdot k_H=k$. 
	By definition $F$ is a subgraph of $G\star H$. If $v\in V$ is
	not incident to an edge in $\fast(M_G,M_H)$ it has degree $0$ in $F$. If $v$ is
	incident to some edge $e \in \fast(M_G,M_H)$, then Lemma \ref{lem:degree-ast} implies
	that $v$ has degree $k = k_G\cdot k_H$ in $F$. Hence, $\fast(M_G,M_H)$ is a
	$k$-matching of $G\star H$. In particular, since neither $M_G$ nor $M_H$ are empty,
	 $\fast(M_G,M_H)\neq\emptyset$.

	Assume now that $\fast(M_G,M_H)$ is a non-empty $k$-matching of $G\star H$ for some
	$M_G\subseteq E_G$ and $M_H\subseteq E_H$. 
	By definition, both $M_G$ and $M_H$ must
	be non-empty. Assume, for contradiction, that $M_G$ is not a $k_G$-matching
	for all $k_G$ that divide $k$. In this case, there must be at least two vertices $g$ and
	$g'$ with different positive degrees in the subgraph $(V_G,M_G) \subseteq G$.
	Assume that $g$ is incident to precisely $k_g$ edges in $M_G$ and that $g'$ is
	incident to precisely $k_{g'}$ edges in $M_G$. W.l.o.g. assume that $1\leq
	k_g<k_{g'}$. Since $\fast(M_G,M_H)$ is a $k$-matching of $G\star H$ and since $M_H$
	is non-empty, there is a vertex $v=(g,h)$ and a vertex $w=(g',h')$ that have both
	degree $k$ in $F$. Let $k_h$ and $k_{h'}$ be the degree of $h$ and $h'$ in
	$(V_H,M_H)\subseteq H$, respectively.
	By similar arguments as in proof of Lemma \ref{lem:degree-ast},
	$k = k_gk_h =	k_{g'}k_{h'}$.
	Since $k_g<k_{g'}$ it must hold $k_h>k_{h'}$.
	In this case, vertex $(g',h)$ must be incident to $k_{g'}k_h>k$ in $\fast(M_G,M_H)$;
	a contradiction.
\end{proof}

Note that in Prop.\ \ref{prop:matching-ast}, if  $M_G$ and $M_H$ are empty, then
$\fast(M_G, M_H)=\emptyset$ is a $k$-matching of $G\star H$. 
However the converse is not satisfied, that is,
if  $\fast(M_G,M_H)=\emptyset$, then one of $M_G$ and
$M_H$ is empty by construction, but not necessarily both.
Thus, we may have that $M_G=\emptyset$ and $M_H$ is some edge set which is 
not a $k_H$-matching, or \emph{vice versa}.


\subsection{$\fcast$-constructions}

Now, we provide a construction of $k$-matchings that is based on $\fast(M_G,M_H)$
   as well as $\MU_G$ and $\MU_H$.

\begin{definition}\label{def:cast}
	For graphs $G$ and $H$ and $M_G\subseteq E_G$ and $M_H\subseteq E_H$, we define
	the following set
   \begin{align*}
   \fcast(M_G, M_H)   \coloneqq &\,    \fast(M_G,M_H) \cup \MU_G \cup \MU_H
   \end{align*}
\end{definition}

Def.\ \ref{def:cast} is illustrated in Fig.\ \ref{fig:fcast-exmpl}.
Recap, the edges in $\fast(M_G,M_H)$ do only exist in the strong, direct and lexicographic product, in general.
While edges in $\fast(M_G,M_H)$ exists in a direct product, edges $\MU_G$, resp.,  $\MU_H$ do not exists in $G \dprod H$.
Hence, the products of interest in this part are the strong and lexicographic product.

\begin{figure}[t]
	\begin{center}
		\includegraphics[width=0.45\textwidth]{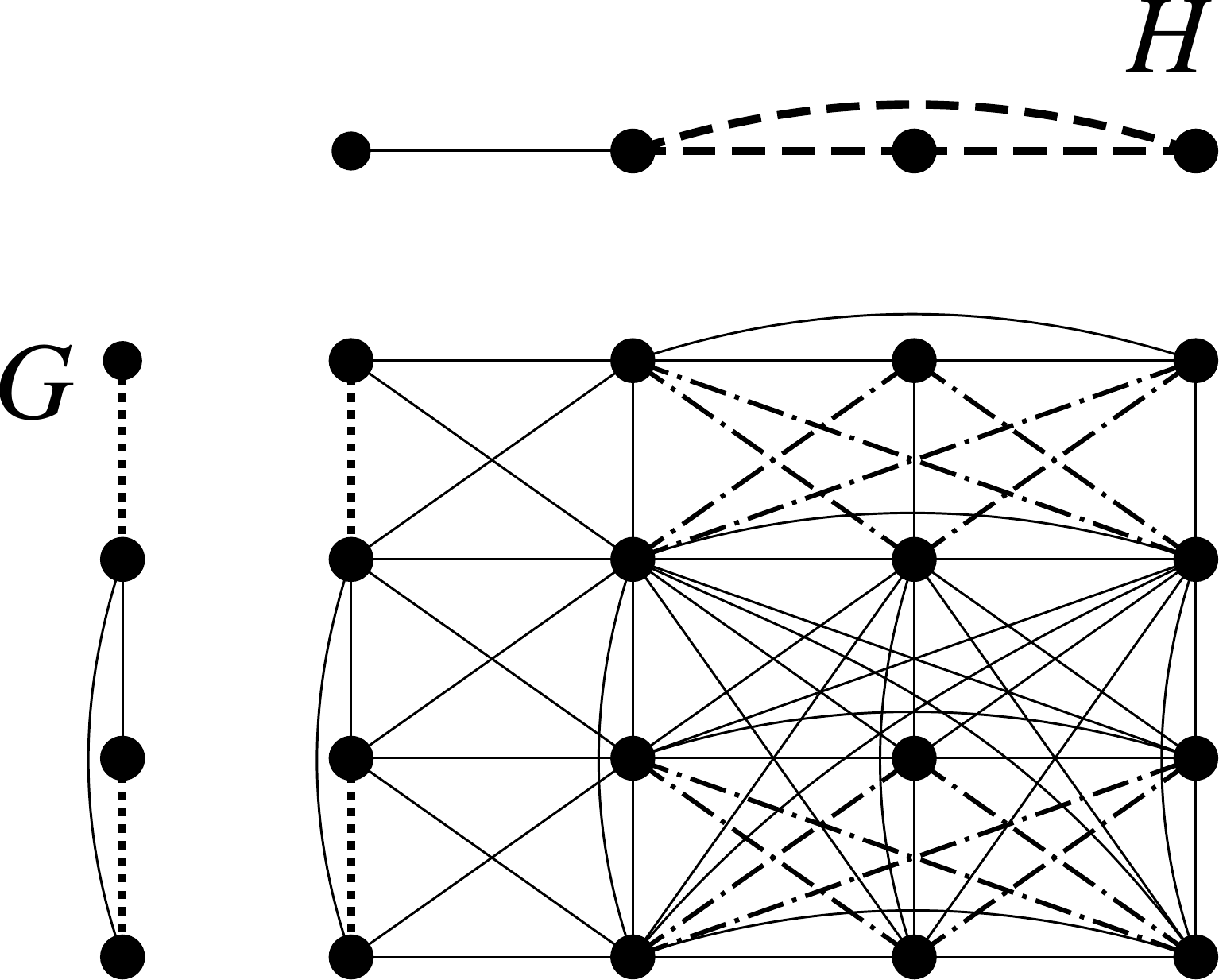}
	\end{center}
	\caption{Shown is the strong product $G\sprod H$ and $\fcast(M_G,M_H)$ for the
	     		 perfect $1$-matching $M_G$ of $G$ and the maximal $2$-matching $M_H$ of $H$.
	     		 The set $\fast(M_G,M_H)$ consists of the dashed-dotted lines and
	     		  and $\MU_H$ of the dashed lines in $G\sprod H$. Here, $\MU_G=\emptyset$.
	     		 Note, $\fcast(M_G,M_H)\subseteq E_{G\sprod H}$ is not a $k$-matching
	     		 since vertices in $V_{G\star H}$ are incident to $1$ or $2$ edges
	     		 in $\fcast(M_G,M_H)$.}
		\label{fig:fcast-exmpl}
\end{figure}

\begin{lemma}\label{lem:ast-sets-disjoint}
	The sets  $\fast(M_G,M_H)$,  $\MU_G$, $\MU_H$ are pairwise vertex disjoint.
\end{lemma}
\begin{proof}
If $\{(g,h),(g',h')\}\in \fast(M_G,M_H)$, then the vertices $g$ and $g'$ are
$M_G$-matched in $G$ and $h$ and $h'$ are $M_H$-matched in $H$.
Consequently, $g,g'\notin U_G$ and $h,h'\notin U_H$.
Hence, none of the edge in $\fast(M_G,M_H)$ can be incident to an edge contained in $\MU_G$ or $\MU_H$.
Therefore,
$\fast(M_G,M_H)$ and $\MU_G$ as well as $\fast(M_G,M_H)$ and $\MU_H$  are vertex disjoint.
Similarly, $\MU_G$ and $\MU_H$ are vertex disjoint, since every edge
$\{(g,h),(g,h')\}\in \MU_G$ satisfies $\{h,h'\}\in M_H$ and thus, $h,h'\notin U_H$
which implies that  none of the edges in $\MU_G$ can be incident to an edge in 
$\MU_H$ and \emph{vice versa}.
\end{proof}

For later reference we provide here a result similar to Lemma \ref{lem:degree-ast}.

\begin{lemma}\label{lem:degree-cast}
Let $\star\in \{\sprod,\lprod\}$ and
$G$ and $H$ be graphs with $k_G$-matching $M_G$ and $k_H$-matching $M_H$, respectively.
Moreover, let $G\star H = (V,E)$ and $v\in V$ be a vertex that is incident to the edge $e \in E$.
If $e\in \MU_G$ (resp., $e\in \MU_H$), then $v$ has degree $k_H$
(resp., $k_G$) in $F=\left(V, \fcast(M_G,M_H)\right)$.
\end{lemma}
\begin{proof}
Let $v=(g,h)\in V$ be a vertex that is incident to the edge $e \in E$
such that $e = \{(g,h),(g',h')\}\in \MU_G$.
Observe first that, by definition $g=g'$ is an $M_G$-unmatched vertex and $\{h,h'\}\in M_H$.
Hence, $h$ must have degree $k_H$ in $(V_H,M_H)$ and $g$ must have degree 0 in $(V_G,M_G)$.
By Lemma \ref{lem:basics-MB-MG},
$(\{g\}\times V_H,\MU_G\cap E_{\lefts{g}{H}})$, $g\in U_G$ and $(V_H,M_H)$ are isomorphic.
Note that distinct copies of $H$ are vertex disjoint and that
$\MU_G$ does not contain  edges connecting vertices of distinct $H$-layers. Moreover, there is
no edge in $\fast(M_G,M_H)$ that is
incident to $v$, since $g$ is unmatched by $M_G$.
Taken the latter two arguments together with Lemma \ref{lem:ast-sets-disjoint},
it follows that $(g,h)$ must have degree $k_H$ in $F$.

By similar arguments, $(g,h)$ must have degree $k_G$ in $F$
provided that $e\in \MU_H$.
\end{proof}

We provide now a characterization for $k$-matchings $\fcast(M_G,M_H)$
		in terms of properties of $M_G$ and $M_H$.

\begin{proposition}\label{prop:matching-cast}
Let $\star\in \{\sprod,\lprod\}$ and  $G$ and $H$ be graphs with $M_G\subseteq E_G$ and $M_H\subseteq E_H$.
Then, $\fcast(M_G,M_H)$ is a $k$-matching of $G\star H$ for some positive integer $k$
if and only if one of the following statements is satisfied:
\begin{enumerate}[itemsep=0.02em]
	\item[M1.] (a) $M_G$ is a perfect $k$-matching of $G$ and $M_H$ is a non-empty $1$-matching of $H$. \\[0.01em]
							 (b) $M_G$ is a non-empty $k$-matching of $G$ and $M_H=\emptyset$. 

	\item[M2.] (a) $M_G$ is a non-empty $1$-matching of $G$ and $M_H$ is a perfect $k$-matching of $H$.\\[0.01em]
						 (b) $M_G=\emptyset$ and  $M_H$ is a non-empty $k$-matching of $H$.

	\item[M3.] $M_G$ is a  $1$-matching of $G$ and $M_H$ is a $1$-matching of $H$ in case $k=1$.
	\item[M4.] $M_G$ is a perfect $k_G$-matching of $G$ and $M_H$ is a perfect $k_H$-matching of $H$, where $k_G, k_H\geq 1$ and $k_G\cdot k_H=k$.
\end{enumerate}
\end{proposition}
\begin{proof}
	In what follows, let $G\star H=(V,E)$ and $F= (V,\fcast(M_G,M_H))$.

	For the \emph{if}-direction assume that $M_G$ and $M_H$ satisfy one of the conditions (M1)-(M4).
	Let us start with Case (M4). In this case, $\MU_G=\MU_H = \emptyset$ and therefore,
	$\fcast(M_G,M_H)  = \fast(M_G,M_H)$. Since neither $M_G$ nor $M_H$ is empty,  
	Prop.\ \ref{prop:matching-ast} implies that $\fcast(M_G,M_H)$ is a $k$-matching of $G \star H$.
	
	Now consider Case (M1). If $M_G$ is a perfect $k$-matching of $G$ we have $U_G=\emptyset$
	and, therefore, $\MU_G=\emptyset$ and if $M_H=\emptyset$, we trivially have $\MU_G=\emptyset$.
	Hence, for (M1.a) and (M1.b), we obtain $\fcast(M_G,M_H)  = \fast(M_G,M_H) \cup \MU_H$.
	By Prop.\ \ref{prop:matching-ast}, $\fast(M_G,M_H)$ is a
	$k$-matching of $G\star H$ for case (M1.a), since $M_G$ and $M_H$ are not empty. If, instead,
	case (M2.b) is satisfied, then $\fast(M_G,M_H)$ is an empty $k$-matching.
	Either way, every vertex $v\in V$ that is incident to
	an edge $e\in \MU_H$ must have degree $k$ in $F$ (cf.\ Lemma \ref{lem:degree-cast}).
	This together with $\fcast(M_G,M_H)  = \fast(M_G,M_H) \cup \MU_H$ implies that
	every vertex in $F$ has either degree $0$ or $k$.
	Thus, $\fcast(M_G,M_H)$ is a $k$-matching of $G\star H$.
	By similar arguments, $\fcast(M_G,M_H)$ is a $k$-matching of $G\star H$
	in Case (M2).
	Finally, in Case (M3), we can apply Lemma \ref{lem:degree-ast}, \ref{lem:ast-sets-disjoint} and \ref{lem:degree-cast}
	to conclude that every vertex in $F$ has either degree $0$ or $1$, which implies
	that $\fcast(M_G,M_H)$ is a (possibly empty) $1$-matching of $G\star H$.

	In summary, whenever $M_G$ and $M_H$ satisfy one of the conditions (M1)-(M4),
	then $\fcast(M_G,M_H)$ is a $k$-matching of $G\star H$.

	For the \emph{only-if}-direction assume that $\fcast(M_G,M_H)$ is a
	$k$-matching of $G \star H$. If $\fcast(M_G,M_H) = \emptyset$, then $M_G=M_H=\emptyset$ and thus,
	they are empty $1$-matching, i.e., (M3) is satisfied. Assume now that $\fcast(M_G,M_H) \neq \emptyset$.
	By Lemma \ref{lem:ast-sets-disjoint} the sets $\fast(M_G,M_H)$,
	$\MU_G$, $\MU_H$ are pairwise vertex disjoint and thus, each of them form  a
	$k$-matching of $G\star H$, independent from the other sets. In particular, $\fast(M_G,M_H)$ is a $k$-matching of
	$G\star H$. In what follows, we distinguish the two cases (a) $\fast(M_G,M_H)=\emptyset$
	and (b) $\fast(M_G,M_H)\neq \emptyset$.
	
	Suppose first Case (a), i.e.,  $\fast(M_G,M_H)$ is an empty $k$-matching. Then precisely one of $M_G$ and $M_H$ is empty,
	while the other is not, since $\fcast(M_G,M_H) \neq \emptyset$. First suppose that $M_G = \emptyset$ 
	and $M_H \neq \emptyset$. Then $\fcast(M_G,M_H)=\MU_G$ since both $\MU_H$ and $\fast(M_G,M_H)$ are empty. 
	Furthermore, every vertex of $G$ is trivially $M_G$-unmatched and hence, $\MU_G \cap E_{\lefts{g}{H}} \neq \emptyset$ for
	each vertex $g\in V_G$. By Lemma\ \ref{lem:basics-MB-MG}, the graph $(\{g\}\times V_H, \MU_G \cap E_{\lefts{g}{H}})$ is isomorphic
	to $(V_H, M_H)$. Since $\fcast(M_G,M_H)=\MU_G$ was assumed to be a $k$-matching and distinct $H$-layers are vertex disjoint, 
	$M_H$ must be a non-empty $k$-matching of $H$, i.e. (M2.b) is satisfied. Instead supposing that $M_G \neq \emptyset$ 
	and $M_H = \emptyset$ we obtain, by similar arguments, that (M1.b) is satisfied.
	
	We continue with Case (b), i.e.,  $\fast(M_G,M_H)$ is a non-empty $k$-matching. Thus, by 
	Prop.\ \ref{prop:matching-ast}, $M_G$ must be a non-empty $k_G$-matching of $G$
	and $M_H$ a non-empty $k_H$-matching of $H$, where $k_G, k_H\geq 1$ and $k_G\cdot k_H=k$.

	We consider now all possible cases for
	$k_G$ and $k_H$: (i) $k_G=k>1$ and $k_H=1$,
	(ii) $k_G=1$ and $k_H=k>1$,
	(iii) $k_G=1$ and $k_H=1$, and
	(iv) $k_G>1$ and $k_H>1$.

 	In Case (i), assume, for contradiction, that $M_G$ is not perfect and thus $U_G\neq \emptyset$.
	Hence, $\MU_G\neq \emptyset$. However, every
	vertex $v$ incident to an edge $e\in \MU_G$ has, by Lemma \ref{lem:degree-cast}, degree
	$1$ in $F$, which implies that $\fcast(M_G,M_H)$ is not a $k$-matching; a
	contradiction. Thus, $M_G$ must be a perfect $k$-matching of $G$. Hence, if Case
	(i) applies then (M1.a) must be satisfied.

	Similarly, Case (ii) implies (M2.a). Moreover, Case (iii) immediately implies
	(M3).

	Finally consider Case (iv). By assumption, $M_G$ is a non-empty $k_G$-matching of $G$ and $M_H$ is a
	non-empty $k_H$-matching of $H$. 
	Assume that one of these matchings, say $M_G$, is not
	perfect. Thus, $U_G\neq \emptyset$. As argued in Case (i), every vertex $v$ incident to an edge $e\in \MU_G$
	must have degree $k_H$ in $F$. Since $k_G\neq 1$, we have $k_H\neq k$. Hence,
	$\fcast(M_G, M_H)$ is not a $k$-matching; a contradiction. Thus, $M_G$ must be a
	perfect $k_G$-matching of $G$. By similar arguments, $M_H$ must be a perfect
	$k_H$-matching of $H$. Hence, Case (iv) implies (M4), which completes the proof.
\end{proof}

For later reference, we provide the following
\begin{lemma}\label{lem:fcast-gh-unmatch}
Let $\star\in\{\sprod,\lprod\}$, $G$ and $H$ be two graphs and
$M_G\subseteq E_G$ and $M_H\subseteq E_H$.
Then, for all $g\in V_G$ and $h\in V_H$,
the following statements are equivalent.
\begin{enumerate}
\item $g\in V_G$ is $M_G$-unmatched in $G$ and $h\in V_H$ is $M_H$-unmatched in $H$,
\item $(g,h)\in V_{G\star H}$ is $\fcast(M_G,M_H)$-unmatched in $G\star H$.
\end{enumerate} 
\end{lemma}
\begin{proof}
	Assume first that $g$ is an $M_G$-unmatched vertex in $G$ and $h$ is an $M_H$-unmatched vertex in $H$.
	For contradiction, suppose that there is an edge $e=\{(g,h),(g',h')\}$ in $\fcast(M_G, M_H)$. By definition, either
	$e\in\fast(M_G,M_H)$, $e\in\MU_G$ or $e\in\MU_H$. In each case, at least one of $\{g,g'\}\in M_G$ or 
	$\{h,h'\}\in M_H$ must hold, contradicting that both $g$ and $h$ are unmatched.
	Hence $(g,h)$ is $\fcast(M_G,M_H)$-unmatched in $G\star H$.
	
	For the converse, assume that $(g,h)$ is unmatched by $\fcast(M_G,M_H)$ in $G\star H$.
	If $g$ would be $M_G$-matched in $G$, then there is an edge $\{g,g'\}\in M_G$.
	Now either $h\in U_H$ or $h$ is $M_H$-matched in $H$ in which case there is an edge $\{h,h'\}\in M_H$.
	If $h\in U_H$, then $\{(g,h),(g',h)\}\in \MU_H\subseteq \fcast(M_G,M_H)$
	and if $\{h,h'\}\in M_H$, then $\{(g,h),(g',h')\}\in \fast(M_G,M_H)\subseteq \fcast(M_G,M_H)$.
	In both cases, $(g,h)$ is $\fcast(M_G,M_H)$-matched; a contradiction.
	Thus, $g$ must be $M_G$-unmatched and hence, $g\in U_G$. Assume now, for contradiction, that $h$
	is $M_H$-matched in $H$ in which case there is an edge $\{h,h'\}\in M_H$.
	Since $g\in U_G$, we have  $\{(g,h),(g,h')\}\in \MU_G\subseteq \fcast(M_G,M_H)$.
	But then $(g,h)$ is $\fcast(M_G,M_H)$-matched; again a contradiction.
	Therefore, for every $\fcast(M_G,M_H)$-unmatched vertex $(g,h)$ in $G\star H$
	the vertex $g$ is $M_G$-unmatched in $G$ and $h$ is $M_H$-unmatched in $H$.
\end{proof}


\section{Well-behaved constructed matchings in products}
\label{sec:wellB}

Assume now that we have some construction for
$k$-matchings $\fM(M_G,M_H)$ of $G\star H$ that is based on a combination of $k_G$-matchings $M_G$ of $G$ and $k_H$-matchings $M_H$ of $H$, 
$\mathpzc{x} \in \{\boxast,\ast,\circledast\}$.s
Clearly $\matchno_k(G\star H)\geq |\fM(M_G,M_H)|$. In what follows we investigate
in more detail, under which conditions such constructed matchings yield a
maximum $k$-matching in $G\star H$, i.e., $\matchno_k(G\star H) = |\fM(M_G,M_H)|$.
Such matchings will be called ``$\mathpzc{x}$-well-behaved''.


\subsection{$\boldsymbol{\boxast}$-well-behaved}

As a direct consequence of Prop.\ \ref{prop:matching-fac}, we obtain the following result.
\begin{corollary}
\label{cor:muLowerBound}
Let $\star\in\{\cprod,\sprod,\lprod\}$. For all graphs $G$ and $H$ with arbitrary
$k$-matching $M_G$ and $M_H$, respectively, it holds that \[ \matchno_k(G\star H)\geq
\max\{|\fbast(M_G,M_H)|, |\fbast(M_H,M_G)| \}\]
\end{corollary}

\begin{figure}[t]
	\begin{center}
		\includegraphics[width=0.45\textwidth]{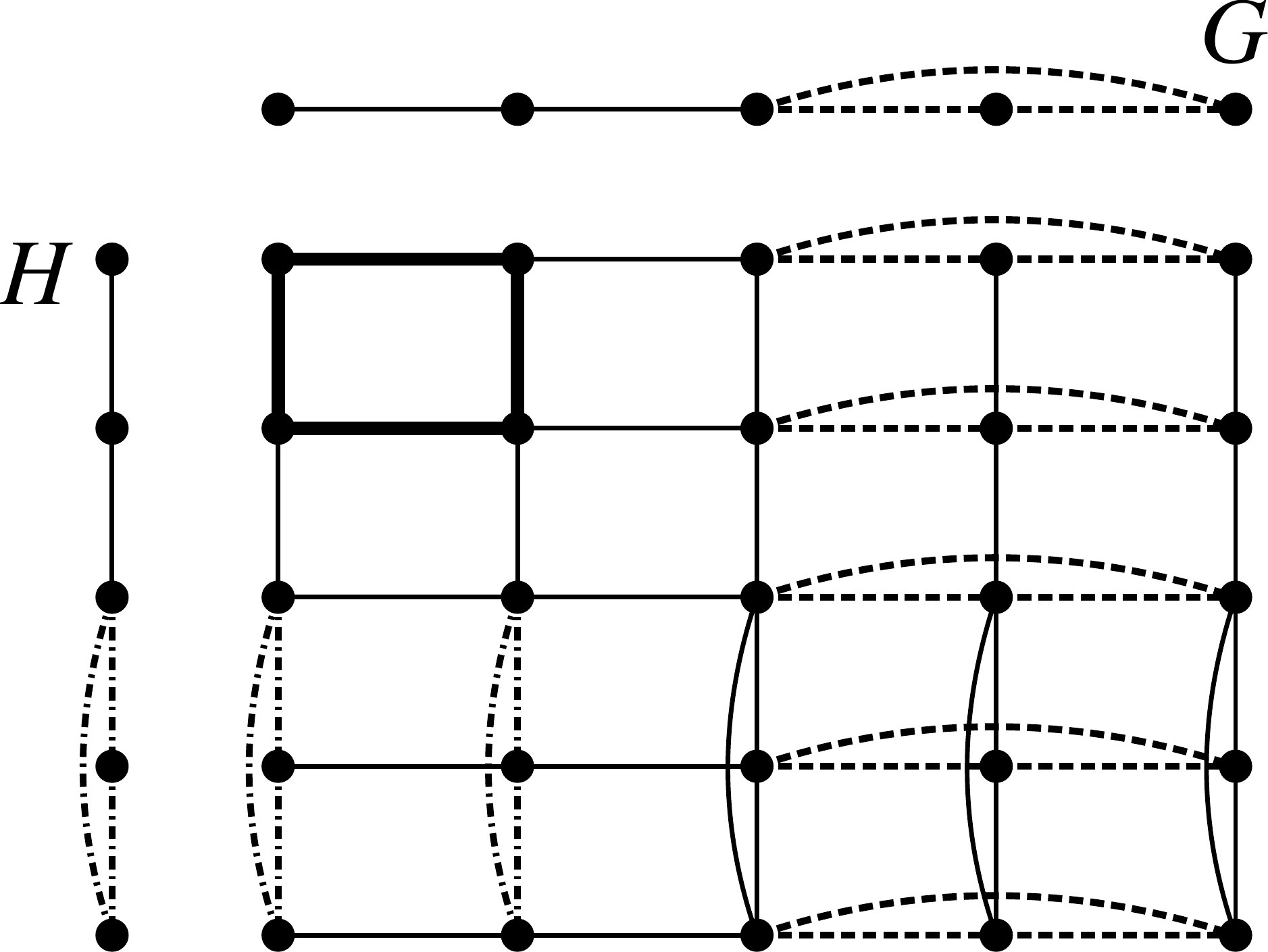}
	\end{center}
	\caption{Shown is the Cartesian product $G\Box H$ and $\fbast(M_G,M_H)$ for the
	         maximal $2$-matchings $M_G$ of $G$ and $M_H$ of $H$, highlighted by different line-styles.
	         $G\Box H$ has a perfect $2$-matching.
	         The subsets $\MB_G$ and $\MU_G$ of $\fbast(M_G,M_H)$  are highlighted by dashed and dashed-dotted lines within $G\cprod H$, respectively.
	         The $4$-cycle  (bold lines) is not part of $\fbast(M_G,M_H)$.
	         Thus, $\fbast(M_G,M_H)$ is a 2-matching   of $G\Box H$ which is not maximal and thus, not perfect.
	         }
		\label{fig:notMaximal}
\end{figure}

Cor.\ \ref{cor:muLowerBound} suggests the following definition.
\begin{definition}\label{def:wellb-Fbast}
Let $\star\in\{\cprod,\sprod,\lprod\}$. The $k$-matching number $\matchno_k(G\star
H)$ of $G\star H$ is \emph{$\boxast$-well-behaved w.r.t.\ $G$ and $H$} if
\[\matchno_k(G\star H) = \max\{|\fbast(M_G,M_H)|, |\fbast(M_H,M_G)|\}\] for some
$k$-matchings $|\fbast(M_G,M_H)|$ and $|\fbast(M_H,M_G)|$ with $M_G\subseteq E_G$ and
$M_H\subseteq E_H$. 
In this case, the $k$-matching $M\in \{\fbast(M_G,M_H),\fbast(M_H,M_G)\}$
of maximum size is also called \emph{$\boxast$-well-behaved}. 

\end{definition}

In other words, the $k$-matching number of $G\star H$ is $\boxast$-well-behaved
w.r.t.\ $G$ and $H$ if one of $\fbast(M_G,M_H)$ and $\fbast(M_H,M_G)$ is a maximum
$k$-matching in $G\star H$.

For the sake of simplicity in upcoming proofs, we assume in Def.\ \ref{def:wellb-Fbast}
that both $\fbast(M_G,M_H)$ and $\fbast(M_H,M_G)$ are $k$-matchings. In fact, this is 
no loss of generality. To see  this assume that $\fbast(M_G,M_H)$ is a $k$-matching
of $G\star H$. Then, $M_G$ is a $k$-matching of $G$
and $M_H$ is a $k$-matching of $H$ (cf.\ Prop. \ref{prop:matching-fac}(2)) or 
$M_G$ is a perfect $k$-matching of $G$ while $M_H$ could be any set (cf.\ Prop. \ref{prop:matching-fac}(1)). 
In case both $M_G$ and $M_H$ are $k$-matchings, 
Prop. \ref{prop:matching-fac}(2) implies that $\fbast(M_H,M_G)$ is a $k$-matching of 
$G\star H$. If $M_G$ is a perfect $k$-matching of $G$, then $M_H$ could be any set
and so we can choose $M_H=\emptyset$. It is easy to verify that, in this case, 
$\fbast(M_H,M_G) = \fbast(M_G,M_H)$ is a $k$-matching of $G\star H$ as well. 

\begin{remark}\label{rem:perfect-fbast}
If $M_G$ is a perfect $k$-matching of $G$ used to construct $\fbast(M_G,M_H)$
we always assume that $M_H=\emptyset$. The same applies if $M_H$ is a perfect 
$k$-matching of $H$ used to construct $\fbast(M_H,M_G)$.
\end{remark}

Furthermore, note that for all products
$\star\in\{\cprod,\sprod,\lprod\}$ the number $\matchno_k(G\star H)$ is
$\boxast$-well-behaved w.r.t.\ $G\star H$ and $K_1$, since $(G\star H)\star
K_1\simeq G\star H$ and thus, $\fbast(M_{G\star H},M_{K_1})  = M_{G\star H}$ yields a
maximum $k$-matching of $G\star H$ provided that $M_{G\star H}$ is a maximum $k$-matching of
$G\star H$.

However, not all $\matchno_k(G\star H)$ are
$\boxast$-well-behaved w.r.t.\ $G$ and $H$ as shown in the following
\begin{example}
Consider the graph $G\cprod
H$ as shown in Fig.\ \ref{fig:inducedMatchings}. In this example, we obtain
$\matchno_1(G') = |\fbast(M_{G'}, M_{K_1})|$ for $G'= G\cprod H$ and thus,
$\matchno_1(G')$ is $\boxast$-well-behaved w.r.t.\ $G'$ and $K_1$ but not w.r.t.\ $G$
and $H$ as shown by the other two $1$-matchings $\fbast(M_G,M_H)$ and $\fbast(M_H,M_G)$.
\end{example}

The latter example implies that the term ``$\boxast$-well-behaved'' heavily depends on
the choice of the factors. 

\begin{example}
Consider the product $G = S_3\Box K_3 \Box P_3$. Then, $\matchno_1(G)$ is $\boxast$-well-behaved
w.r.t.\ $S_3\Box K_3$ and $P_3$ but not w.r.t.\ $S_3$ and $K_3 \Box P_3$. To see this,
choose for $S_3\Box K_3$ the perfect $1$-matching as in Fig.\
\ref{fig:inducedMatchings}(right). Hence, $\MU_{S_3\Box K_3}=\emptyset$ and
$\fbast(M_{S_3\Box P_3}, M_{K_3}) = \MB_{S_3\Box K_3}$ yields a perfect $1$-matching in $G$, that is,
$\matchno_1(G) = |\fbast(M_{S_3\Box P_3}, M_{K_3})|$. In contrast, neither $S_3$ nor
$K_3 \Box P_3$ has a perfect $1$-matching (the latter since it has an odd number of vertices).
It is now easy to verify that
$\fbast(M_{S_3\Box P_3}, M_{K_3})$ yields a $1$-matching of $G$ but no perfect one, although one exists.
Therefore, $\matchno_1(G) > |\fbast(M_{K_3\Box P_3}, M_{S_3})|$ and, consequently,
$\matchno_1(G)$ is not $\boxast$-well-behaved w.r.t.\ $K_3 \Box P_3$ and $S_3$.
\end{example}

We now investigate the cardinalities $|\fbast(M_G,M_H)|$ and $|\fbast(M_H,M_G)|$ in
some more detail. Intriguingly, $\boxast$ is ``commutative'' for $k$-matchings $M_H$ and
$M_G$ in the sense that the cardinalities of the sets $\fbast(M_G,M_H)$ and $\fbast(M_H,M_G)$
are equal.

\begin{proposition}
\label{prop:maxexprequal}
Let $G$ and $H$ be graphs of order $n_G$ and $n_H$, respectively. For all
$k$-matchings $M_G$ of $G$ and $M_H$ of $H$, we have 
\begin{align*}
|\fbast(M_G,M_H)| &= |M_G| n_H+|M_H|u_G  \text{ and } \\
|\fbast(M_G,M_H)| &= |\fbast(M_H,M_G)| = \frac{k}{2} \left(n_Gn_H-u_Gu_H \right),
\end{align*}
where $u_G$ and $u_H$ denote the number of
$M_G$-unmatched and $M_H$-unmatched vertices in $G$ and $H$, respectively.
\end{proposition}
\begin{proof}
By definition and Lemma\ \ref{lem:basics-MB-MG}, $|\fbast(M_G,M_H)| = |\MB_G|+ |\MU_G| = |M_G| n_H+|M_H|u_G$. By
Lemma~\ref{lem:muFormula}, $|M_G|=k(n_G-u_G)/2$ and $|M_H|=k(n_H-u_H)/2$. Thus, we
have
\begin{align*}
|\fbast(M_G,M_H)| =|M_G| n_H+|M_H|u_G&=\frac{k(n_G-u_G)}{2}\cdot
n_H+\frac{k(n_H-u_H)}{2}\cdot u_G\\ &= \frac{k}{2}
\left(n_Gn_H-u_Gn_H+n_Hu_G-u_Hu_G\right)\\ &= \frac{k}{2} \left(n_Gn_H-u_Hu_G\right).
\end{align*}
Similarly, one shows that $|\fbast(M_H,M_G)| = |M_H| n_G+|M_G|u_H $
satisfies $|\fbast(M_H,M_G)| = \frac{k}{2} \left(n_Gn_H-u_Hu_G\right)$ and thus,
$|\fbast(M_G,M_H)| = |\fbast(M_H,M_G)| $.
\end{proof}

Based on Proposition \ref{prop:maxexprequal} we obtain

\begin{proposition}\label{prop:cardi}
Let $\star\in\{\cprod,\sprod,\lprod\}$. If the $k$-matching number $\matchno_k(G\star
H)$ of $G\star H$ is $\boxast$-well-behaved w.r.t.\ $G$ and $H$, then
\[\matchno_k(G\star H) = |\fbast(M_G,M_H)| = |\fbast(M_H,M_G)|\]
for some $k$-matching $M_G$ and $M_H$ of $G$ and $H$, respectively.
In this case, $M_G$ and $M_H$ must be maximum $k$-matchings of $G$ and $H$, respectively.
\end{proposition}
\begin{proof}
	By definition and Proposition \ref{prop:maxexprequal}, if $\matchno_k(G\star H)$ is
	$\boxast$-well-behaved w.r.t.\ $G$ and $H$, then $\matchno_k(G\star H) =
	|\fbast(M_G,M_H)| = |\fbast(M_H,M_G)|$ for some $k$-matchings
	$|\fbast(M_G,M_H)|$ and $|\fbast(M_H,M_G)|$
	with $M_G\subseteq E_G$ and $M_H\subseteq E_H$. Since both $\fbast(M_G,M_H)$
	and $\fbast(M_H,M_G)$ are assumed to be (possibly empty) $k$-matchings,
	Cor.\ \ref{cor:matching-fac} implies that $M_G$ is a $k$-matching of $G$
	and $M_H$ a $k$-matching of $H$.

	It remains to show that these $k$-matchings $M_G$ and $M_H$ are also maximum
	$k$-matchings of $G$ and $H$, respectively. Assume, for contradiction, that $M_G$
	is not a maximum $k$-matching of $G$ and let $M^*$ be a maximum $k$-matching of
	$G$. Thus, $|M^*|>|M_G|$ and, by Lemma \ref{lem:muFormula}, $u^*< u_G$ where $u^*$
	and $u_G$ denote the number of $M^*$- and $M_G$-unmatched vertices in $G$,
	respectively. Put $n=|V_G||V_H|$ and let $u_H$ be the number of $M_H$-unmatched
	vertices in $H$. By Prop.\ \ref{prop:matching-fac}, $|\fbast(M^*, M_{H})|$ is a
	$k$-matching of $G\star H$. By Proposition \ref{prop:maxexprequal},
	$|\fbast(M^*, M_{H})|= \frac{k}{2} \left(n-u^*u_H \right)> \frac{k}{2} \left(n-u_Gu_H \right) =
	|\fbast(M_G, M_{H})|$. Taken the latter two arguments together,
	$\matchno_k(G\star H)	\neq |\fbast(M_G,M_H)|$; a contradiction. Thus, $M_G$ must be a maximum
	$k$-matching of $G$. By similar arguments, $M_H$ is a maximum $k$-matching of $H$.
\end{proof}

We are now in the position to characterize $\boxast$-well-behaved $k$-matching numbers.

\begin{theorem}\label{thm:2ndChar}
	Let $\star\in\{\cprod,\sprod,\lprod\}$ and $G$ and $H$ be two graphs
	of order $n_G$ and $n_H$, respectively.
	The following statements are equivalent, where $u_G$ and $u_H$ denote
	the $M_G$- and $M_H$-unmatched vertices in $G$ and $H$, respectively,
	for $M_G\subseteq E_G$ and $M_H\subseteq E_H$.
	\begin{enumerate}[itemsep=0.02em]
		\item $\matchno_k(G\star H) $ is $\boxast$-well-behaved w.r.t.\ $G$ and $H$
		\item There is some  $k$-matching $M_G$ of $G$ and
					some $k$-matching $M_H$ of $H$ such that \\
					$\unmatchno_k(G\star H) = u_{G}u_{H}$
	\end{enumerate}
\end{theorem}
\begin{proof}
If $\matchno_k(G\star H) $ is $\boxast$-well-behaved w.r.t.\ $G$ and $H$, then there
are, by Prop.\ \ref{prop:cardi}, some maximum $k$-matchings $M_G$ and $M_H$ of $G$ and
$H$, resp, that satisfy $\matchno_k(G\star H) = |\fbast(M_G,M_H)|$. 
 By Prop.\ \ref{prop:maxexprequal}, $\matchno_k(G\star H) =
|\fbast(M_G,M_H)| = \frac k2 \left(n_Gn_H- u_Gu_H\right)$. By
Lemma~\ref{lem:muFormula}, $\matchno_k(G\star H)=\frac{k}{2} (n_{G\star
H}-\unmatchno_k(G\star H ))=\frac{k}{2} (n_G n_H-\unmatchno_k(G\star H))$, and thus,
$\unmatchno_k(G\star H) = u_{G}u_{H}$.

Now assume that $M_G$ and $M_H$ are $k$-matchings of $G$ and $H$, respectively, such
that $\unmatchno_k(G\star H) = u_{G}u_{H}$. This and Prop.\ \ref{prop:cardi} imply that
$\matchno_k(G\star H)=\frac{k}{2} (n_{G\star H}-\unmatchno_k(G\star H ))=\frac{k}{2}
(n_G n_H-\unmatchno_k(G\star H)) = \frac{k}{2} \left(n_Gn_H- u_Gu_H\right) =
|\fbast(M_G,M_H)|$. Thus, $\matchno_k(G\star H) $ is $\boxast$-well-behaved w.r.t.\
$G$ and $H$.
\end{proof}

As it turns out, $\boxast$-well-behaved $k$-matching numbers can be
characterized  entirely  in terms of the numbers $\matchno_k(G)$, $\matchno_k(H)$,
$\unmatchno_k(G)$, $\unmatchno_k(H)$ and the size of $G$ and $H$.

\begin{theorem}
\label{thm:wellB-maxM}
Let $\star\in\{\cprod,\sprod,\lprod\}$ and
let $G$ and $H$ be graphs of order $n_G$ and $n_H$, respectively.
Furthermore, let $\mathcal{M}_G$ and $\mathcal{M}_H$ denote the set
of all maximum $k$-matchings of $G$ and $H$, respectively.
The following statements are equivalent:
\begin{enumerate}[itemsep=0.02em]
	\item $\matchno_k(G\star H)$ is $\boxast$-well-behaved w.r.t.\ $G$ and $H$.
	\item $\fbast(M_G,M_H)$ is a maximum $k$-matching of $G\star H$
				for all $M_G\in \mathcal{M}_G$  and $M_H\in \mathcal{M}_H$.
	\item $\fbast(M'_H, M'_G)$ is a maximum $k$-matching in $G\star H$
				for all $M'_H\in \mathcal{M}_H$ and $M'_G\in \mathcal{M}_G$. 
	\item	$\matchno_k(G\star H) = \matchno_k(G) n_H+\matchno_k(H)\unmatchno_k(G)$.
  \item $\matchno_k(G\star H)  = \matchno_k(H) n_G+\matchno_k(G)\unmatchno_k(H)$.
	\item $\matchno_k(G\star H)= \frac{k}{2} \left(n_Gn_H-\unmatchno_k(G)\unmatchno_k(H) \right)$.
	\item $\unmatchno_k(G\star H)=\unmatchno_k(G)\unmatchno_k(H)$.
\end{enumerate}
\end{theorem}
\begin{proof}

Assume first that $\matchno_k(G\star H)$ is $\boxast$-well-behaved w.r.t.\ $G$ and $H$. By
Prop.\ \ref{prop:cardi}, $\matchno_k(G\star H) = |\fbast(M_G,M_H)| =
|\fbast(M_{H} , M_{G})|$ for some $M_G\in \mathcal{M}_G$ and $M_H\in \mathcal{M}_H$.
 For all $M_G,
\,M'_G\in \mathcal{M}_G$ we have that $u_G=u'_G$, where $u_G$ and $u'_G$ denote the
number of $M_G$-unmatched and $M_G'$-unmatched vertices, respectively. Similarly,
$u_H=u'_H$ for all $M_H, M'_H\in \mathcal{M}_H$. This together with Prop.\
\ref{prop:maxexprequal} implies that $\matchno_k(G\star H) = |\fbast(M_G,M_H)| =
\frac{k}{2} \left(n_Gn_H-u_Gu_H \right) = \frac{k}{2} \left(n_Gn_H-u'_Gu'_H \right) =
|\fbast(M'_G , M'_H)|$. Therefore, $\fbast(M'_G , M'_H)$ is a maximum $k$-matching in
$G\star H$. Hence, Condition (1) implies (2) and, by similar arguments, (3).

Reusing the latter arguments together with Prop.\ \ref{prop:maxexprequal}
shows that  Condition (2) and (3) are equivalent since
$\fbast(M_G,M_H) $ and $\fbast(M'_H, M'_G)$ are of the same size
for all  $M_G,M'_G\in \mathcal{M}_G$ and all $M_H, M'_H\in \mathcal{M}_H$.
Moreover, if $\fbast(M_G,M_H)$ is a maximum $k$-matching of $G\star H$, then
$\matchno_k(G\star H)=|\fbast(M_G,M_H)|$ and thus, (2) implies (1).
Therefore, Condition (1), (2) and (3) are equivalent.

Note, for all $M_G\in \mathcal{M}_G$  and $M_H\in \mathcal{M}_H$ we have $|M_G| =  \matchno_k(G)$
and $|M_H| =  \matchno_k(H)$ and thus, by Prop.\ \ref{prop:maxexprequal},
\begin{align*}\label{eq:boxast-eq}
|\fbast(M_G,M_H)| &= \matchno_k(G) n_H+\matchno_k(H)\unmatchno_k(G) \tag{a} \\
									&  = \matchno_k(H) n_G+\matchno_k(G)\unmatchno_k(H)=|\fbast(M_H, M_G)|.
\end{align*}
Taken this together with Prop.\ \ref{prop:cardi} and the equivalence of Condition (1), (2) and (3),
Condition (1) implies Condition (4) and (5).
Condition (4) and (5) taken together with Eq.\ \eqref{eq:boxast-eq} imply Condition (2) and (3), respectively.
Moreover, Eq.\ \eqref{eq:boxast-eq} together
with Prop.\ \ref{prop:maxexprequal}, imply that Condition (4), (5) and (6) are equivalent.
Therefore, Condition (1) - (6) are equivalent.

We finally show that Condition (6) and (7) are equivalent.
To this end, observe first that Lemma~\ref{lem:muFormula} implies that
$\matchno_k(G\star H)=\frac{k}{2} (n_{G\star H}-\unmatchno_k(G\star H ))=\frac{k}{2} (n_G n_H-\unmatchno_k(G\star H)) $
is always satisfied.
Hence, Condition (6) $\matchno_k(G\star H)= \frac k2 \left(n_Gn_H-\unmatchno_k(H)\unmatchno_k(G) \right)$
is satisfied if and only $\unmatchno_k(G\star H)=\unmatchno_k(G)\unmatchno_k(H)$.
\end{proof}

\begin{figure}[t]
	\begin{center}
		\includegraphics[width=.9\textwidth]{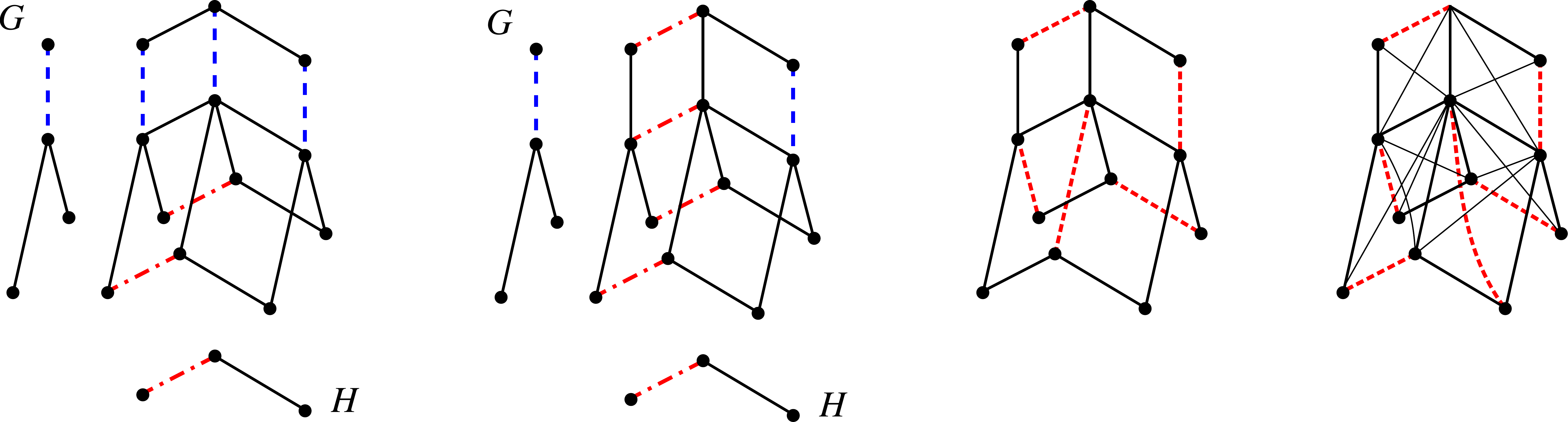}
	\end{center}
	\caption{Shown is the Cartesian product $G\cprod H$ (1.-3. from left)
					 and the strong product $G\sprod H$ (right).
					 Moreover, the sets $\fcast(M_G,M_H)$ (left)
						and $\fcast(M_H,M_G)$ (2nd from left) for the
	     		 maximum $1$-matchings $M_G$ of $G$ and $M_H$ of $H$
	     		 (highlighted by dashed lines) are provided. 
	     		 A further $1$-matching of  $G\Box H$ is shown 3rd from left.
	     		 All $1$-matchings are maximum $1$-matchings of $G\Box H$. In particular, for $k=1$,
				 		 $\matchno_k(G\cprod H)$ is $\boxast$-well-behaved w.r.t.\ $G$ and $H$,
				 		since $\matchno_k(G\cprod H) = |\fbast(M_G,M_H)| = |\fbast(M_H,M_G)| =5$ but also because
						$\unmatchno_k(G\cprod H) = u_{G}u_{H}$ with $u_G$ and $u_H$ being
						the $M_G$- and $M_H$-unmatched vertices in $G$ and $H$, respectively (cf.\ Thm.\ \ref{thm:2ndChar}).
						However, $\matchno_k(G\sprod H) = 6$. Since  
						$2= u_{G}u_{H}\neq 0 = \unmatchno_k(G\sprod H)$, Thm.\ \ref{thm:2ndChar} implies that 
					  $\matchno_k(G\sprod H)$	is not $\boxast$-well-behaved w.r.t.\ $G$ and $H$. 
						}
		\label{fig:exmpl2-new}
\end{figure}

Note that $\boxast$-well-behavedness heavily depends on the chosen products as shown in the following result
   and the subsequent example.

\begin{proposition} \label{prop:wb-implications}
Let $G$ and $H$ be two graph. 
If $\matchno_k(G \lprod H)$ is  $\boxast$-well-behaved w.r.t.\ $G$ and $H$,
then $\matchno_k(G \sprod H)$ is  $\boxast$-well-behaved w.r.t.\ $G$ and $H$,
which implies that $\matchno_k(G \cprod H)$ is  $\boxast$-well-behaved w.r.t.\ $G$ and $H$. 
\end{proposition}
\begin{proof}
Assume that $\matchno_k(G \lprod H)$ is  $\boxast$-well-behaved w.r.t.\ $G$ and $H$
and thus,  $\matchno_k(G \lprod H) = \fbast(M_G,M_H)$ for some $k$-matching 
$\fbast(M_G,M_H)$ of $G \lprod H$. Since $\fbast(M_G,M_H)$ consists of Cartesian
edges only, we have $\fbast(M_G,M_H)\subseteq E_{G\sprod H}$ and $\fbast(M_G,M_H)\subseteq E_{G\cprod H}$. 
This together with  $G \cprod H\subseteq G \sprod H \subseteq G \lprod H$ implies that
$|\fbast(M_G,M_H)|\leq \matchno_k(G \cprod H)\leq \matchno_k(G \sprod H) \leq \matchno_k(G \lprod H)$
and thus, we obtain equality in the latter equation. 
By similar, arguments one shows  $\boxast$-well-behavedness of  $\matchno_k(G \sprod H)$
implies that $\matchno_k(G \cprod H)$ is  $\boxast$-well-behaved w.r.t.\ $G$ and $H$. 
\end{proof}

The converse of Prop.\ \ref{prop:wb-implications} is not satisfied in general. 
		Fig.\ \ref{fig:exmpl2-new} provides an example for which $\matchno_k(G \cprod H)$ is  
		$\boxast$-well-behaved w.r.t.\ $G$ and $H$ but not $\matchno_k(G \sprod H)$.

The following two results are a mild generalization of results
that have been established for for the Cartesian and strong product and for $1$-matchings
(cf.\ e.g.\ \cite[Lemma 7 and Prop.\ 17]{AlmeidaEtAl2013} or \cite[Thm.\ 4]{Kotzig1979}).

\begin{corollary}\label{cor:perfect-k}
	Let $\star\in\{\cprod,\sprod,\lprod\}$.
	If $G$ has a perfect $k$-matching, then $G \star H$ has a perfect $k$-matching for all graphs $H$.
\end{corollary}
\begin{proof}
Let  $n_G$ and $n_H$ be the order of $G$ and $H$, respectively.
Let $M_H$ be any $k$-matching of $H$.
If $G$ has a perfect $k$-matching, then $\matchno_k(G) = k\frac{n_G}{2}$ and
$\unmatchno_k(G)=0$.
This together with Prop.\ \ref{prop:maxexprequal} implies that
$|\fbast(M_G,M_H)| =  \matchno_k(G) n_H = k\frac{n_G n_H}{2} = k \frac{|V_{G\star H}|}{2}$,
the size of a perfect $k$-matching in $G\star H$, for all graphs $H$.
\end{proof}

\begin{corollary}
	Let $\star\in\{\cprod,\sprod,\lprod\}$.
	If $G$ and $H$ have near-perfect $k$-matchings, then $G \star H$ has a near-perfect $k$-matching.
\end{corollary}
\begin{proof}
Let  $n_G$ and $n_H$ be the order,
and $M_G$ and $M_H$ be near-perfect $k$-matchings, of $G$ and $H$, respectively.
Hence, $u_G = u_H =1$, with $u_G$ and $u_H$ being $M_G$- and $M_H$-unmatched vertices, respectively.
By Lemma \ref{lem:muFormula}, $|M_G| =k\cdot \frac{n_G-1}{2}$
and $|M_H| =k\cdot \frac{n_H-1}{2}$.
Prop.\ \ref{prop:maxexprequal} implies that $|\fbast(M_G,M_H)| =  |M_G|n_H + |M_H|u_G$.
Taken the latter arguments together, we obtain
$|\fbast(M_G,M_H)| = k \frac{(n_G-1)n_H}{2} + k \frac{n_H-1}{2} = k  \frac{n_Gn_H-1}{2}
= k\frac{|V_{G\star H}|-1}{2}$,
the size of a  near-perfect $k$-matching in $G\star H$.
\end{proof}


\subsection{$\boldsymbol{\circledast}$-well-behaved}

In what follows we will make frequent use of the condition (M1), (M2), (M3) and (M4) as specified in Proposition\ \ref{prop:matching-cast}.
We start with specifying the cardinality $|\fcast(M_G,M_H)|$.

\begin{proposition}\label{prop:size-cast}
 Let $\star\in \{\sprod,\lprod\}$ and
  $G$ and $H$ be two graphs of size $n_G$ and $n_H$, respectively. Moreover, let
 $M_G\subseteq E_G$ and $M_H\subseteq E_H$ and denote the number of $M_G$-unmatched vertices in $G$
 and $M_H$-unmatched vertices in $H$ by $u_G$ and $u_H$, respectively.
			If 	$\fcast(M_G,M_H)$  is a $k$-matching of $G\star H$, then
			 \[|\fcast(M_G,M_H)| = \frac{k}{2} \left(n_Gn_H-u_Gu_H \right). \]
Moreover, if $M_G$ and $M_H$ satisfy one of the Cases (M1.a), (M2.a) and (M4),
then $\fcast(M_G,M_H)$  is a perfect $k$-matching of $G\star H$ and
			 \[|\fcast(M_G,M_H)| = \frac{k}{2}n_Gn_H.\]
\end{proposition}
\begin{proof}
	 Let $G$ and $H$ be of order $n_G$ and $n_H$, respectively and $M_G\subseteq E_G$ and $M_H\subseteq E_H$.
	 Assume that $\fcast(M_G,M_H)$  is a $k$-matching of $G\star H$.
			By Prop.\	\ref{prop:matching-cast}, the sets $M_G$ and $M_H$ must satisfy one of the Conditions (M1)-(M4).
		Moreover, by Lemma
	\ref{lem:ast-sets-disjoint}, the sets $\fast(M_G,M_H)$, $\MU_G$, $\MU_H$ are pairwise
	vertex disjoint and, therefore, $|\fcast(M_G,M_H)| = |\fast(M_G,M_H)| + |\MU_G| +
	|\MU_H|$.	In addition, $|\fast(M_G,M_H)|=2|M_G||M_H|$ since for each pair $\{g,g'\}\in M_G$
	and $\{h,h'\}\in M_H$  the two edges $\{(g,h),(g',h')\}$ and $\{(g,h'),(g',h)\}$ are contained in $\fast(M_G,M_H)$.

	Assume now that $M_G$ is a $k_G$-matching of $G$ and $M_H$ is a $k_H$-matching of $H$.
	Let  $u_G$ and $u_H$ be the number of	$M_G$- and $M_H$-unmatched vertices in $G$ and $H$, respectively.
	By construction, $|\MU_G| = u_G|M_H|$ and $|\MU_H| = u_H|M_G|$.
	Taken the latter arguments together with Lemma \ref{lem:muFormula} we obtain
  \[|\fast(M_G,M_H)| = 2\cdot\frac{k_G (n_G-u_G)}{2}\cdot\frac{k_H(n_H-u_H)}{2} \]
  and
	\[|\MU_G| = u_G \cdot \frac{k_H(n_H-u_H)}{2} \text{ \ and \  }
		|\MU_H| = u_H \cdot \frac{k_G(n_G-u_G)}{2}.\]

	Suppose first that (M1.a) is satisfied, and thus, that $M_G$ is a perfect
	$k$-matching of $G$ and $M_H$ is a non-empty $1$-matching of $H$.
	Hence, $u_G=0$ and thus, $\MU_G=\emptyset$.
	Therefore, $|\fcast(M_G,M_H)|$ reduces to
	\begin{align*}
	|\fcast(M_G,M_H)| & =  |\fast(M_G,M_H)| + |\MU_H|  \\
											&	 = 2\left( \frac{k\cdot n_G}{2}\right)\left(\frac{n_H-u_H}{2}\right) +  u_H \frac{k\cdot n_G}{2}  \\
											&	=	\frac{kn_Gn_H-kn_Gu_H}{2} +  \frac{kn_Gu_H}{2}= \frac{k}{2} n_Gn_H.
	\end{align*}
	By analogous arguments, we obtain the desired equality if (M2.a) is satisfied.
	Note, in both cases (M1.a) and (M2.a), we have $|\fcast(M_G,M_H)| =  \frac{k}{2} n_Gn_H = \frac{k}{2} n_{G\star H}$
 	and Lemma \ref{lem:muFormula} implies that $\fcast(M_G,M_H)$ must be a perfect $k$-matching of $G\star H$.

	Now assume that Case (M1.b) is satisfied. Thus, $M_H=\emptyset$ and therefore,
			$\MU_G = \emptyset$ and $n_H=u_H$.
	Hence, $|\fcast(M_G,M_H)|$ reduces to
	\begin{align*}
	|\fcast(M_G,M_H)| & =  |\fast(M_G,M_H)| + |\MU_H|  \\
											&	 = 2\cdot |M_G|\cdot 0 +  u_H \frac{k(n_G-u_G)}{2}
												=	\frac{k}{2}\left(n_Gn_H-u_Gu_H\right).
	\end{align*}
	where we employed for the last equality $n_H=u_H$.
		By analogous arguments, we obtain the desired equality if (M2.b) is satisfied.

	Now assume that Case (M3) is satisfied and thus,
	$M_G$ is a $1$-matching of $G$ and $M_H$ is a $1$-matching of $H$. We have
	\begin{align*}
	|\fcast(M_G,M_H)| & = 2|M_G||M_H| + u_H|M_G| + u_G|M_H| \\
											 & = |M_G|(2|M_H| + u_H) + u_H|M_G|\\
											 & = |M_G|n_H  + u_H|M_G| =  |\fbast(M_G,M_H)| \\
											 & = \frac{k}{2} \left(n_Gn_H-u_Gu_H \right),
	\end{align*}
	where we employed for the 2nd last equality $n_H = 2|M_H| + u_H$ (which is satisfied
	by Lemma \ref{lem:muFormula} for $1$-matchings) and
	for the last equality  we used Prop.\ \ref{prop:maxexprequal}.

	Finally, consider Case (M4).
	Thus, $M_G$ is a perfect $k_G$-matching of $G$ and $M_H$ is a perfect $k_H$-matching of $H$
	with $k=k_Gk_H$.
	Therefore, $u_G=u_H=0$ and $\MU_G=\MU_H=\emptyset$.
	Hence, we obtain
	\begin{align*}
	|\fcast(M_G,M_H)| & = |\fast(M_G,M_H)|
											  = 2\frac{k_G n_G}{2}\frac{k_H n_H}{2}
											  = 	k_Gk_H \frac{n_Gn_H}{2} =  \frac{k}{2} n_Gn_H.
	\end{align*}
  Thus, $|\fcast(M_G,M_H)| =  \frac{k}{2} n_Gn_H = \frac{k}{2} n_{G\star H}$
 	and Lemma \ref{lem:muFormula} implies that $\fcast(M_G,M_H)$ must be a perfect $k$-matching of $G\star H$.

 	By the latter arguments, for a perfect $k$-matching $\fcast(M_G,M_H)$ of $G\star H$ the sets $M_G$ and $M_H$
 	must satisfy one of the conditions (M1.a), (M2.a) or (M4), in which case, $|\fcast(M_G,M_H)| =  \frac{k}{2} n_Gn_H$
 	is always satisfied.
\end{proof}

\begin{definition}
Let $\star\in\{\sprod,\lprod\}$. The $k$-matching number $\matchno_k(G\star
H)$ of $G\star H$ is \emph{$\circledast$-well-behaved w.r.t.\ $G$ and $H$}
if  \[\matchno_k(G\star H) = |\fcast(M_G,M_H)|\]
for some $k$-matching $\fcast(M_G,M_H)$
with $M_G\subseteq E_G$ and $M_H\subseteq E_H$.
In this case, the $k$-matching $\fcast(M_G,M_H)$
 is also called \emph{$\circledast$-well-behaved}.
\end{definition}

We provide now characterization of $\circledast$-well-behaved $k$-matching numbers.

\begin{theorem}\label{thm:cast-wellb-new}
	Let $\star\in\{\sprod,\lprod\}$.
	The following two statements are equivalent for all graphs $G$ and $H$ and every integer $k\geq 1$.
	\begin{enumerate}[itemsep=0.02em]
		\item $\matchno_k(G\star H)$ is $\circledast$-well-behaved w.r.t.\ $G$ and $H$
		\item There is some $k_G$-matching $M_G$ of $G$ and
					some $k_H$-matching $M_H$ of $H$
					such that $M_G$ and $M_H$ satisfy one of the Conditions (M1), (M2), (M3) and (M4)
					and
					$\unmatchno_k(G\star H) = u_{G}u_{H}$
					with $u_G$ and $u_H$ being the $M_G$- and $M_H$-unmatched vertices in $G$ and $H$, respectively.
	\end{enumerate}
\end{theorem}
\begin{proof}
	Assume that $\matchno_k(G\star H)$ is $\circledast$-well-behaved w.r.t.\ $G$ and $H$.
	Thus, there is some $k$-matching $\fcast(M_G,M_H)$ of size $\matchno_k(G\star H)$.
	Prop.\ \ref{prop:matching-cast} implies that $M_G$ and $M_H$
	must satisfy one of the Conditions (M1), (M2), (M3) and (M4).
	By   Prop.\ \ref{prop:size-cast}, $|\fcast(M_G,M_H)| = \frac{k}{2}
  \left(n_Gn_H-u_Gu_H \right)$.
   By Lemma~\ref{lem:muFormula},
   $\matchno_k(G\star H)=\frac{k}{2} (n_{G\star H}-\unmatchno_k(G\star H ))=\frac{k}{2} (n_G n_H-\unmatchno_k(G\star H))$,
  and thus, $\unmatchno_k(G\star H) = u_{G}u_{H}$.

	Now assume that there is a $k_G$-matching $M_G$ of $G$ and a $k_H$-matching $M_H$ of $H$
	such that $M_G$ and $M_H$ satisfy one of the Conditions (M1), (M2), (M3) and (M4)
	and $\unmatchno_k(G\star H) = u_{G}u_{H}$.
	By Prop.\ \ref{prop:matching-cast},
	$\fcast(M_G,M_H)$ is a $k$-matching of $G\star H$.
	This together with $\unmatchno_k(G\star H) = u_{G}u_{H}$, Lemma \ref{lem:muFormula}  and  Prop.\ \ref{prop:size-cast} implies that
	$|\fcast(M_G,M_H)| = \frac{k}{2} \left(n_Gn_H- u_Gu_H\right) =\frac{k}{2} (n_G n_H-\unmatchno_k(G\star H)) =\frac{k}{2} (n_{G\star H}-\unmatchno_k(G\star H )) = \matchno_k(G\star H)$.
	Thus, $\matchno_k(G\star H) $ is $\boxast$-well-behaved w.r.t.\ $G$ and $H$.
	\end{proof}

\begin{proposition}
 Let $\star\in \{\sprod,\lprod\}$. If there are $k$-matchings $M_G$ of $G$ and $M_H$ of $H$ such that
	 one of  $M_G$ or $M_H$ is empty and   $\matchno_k(G\star H) = |\fbast(M_G,M_H)|$, then
 $\matchno_k(G\star H)$ is   $\circledast$-well-behaved  w.r.t.\ $G$ and $H$.
\end{proposition}
\begin{proof}
 	Let $M_G$ or $M_H$ be $k$-matchings  $G$ and $H$, respectively, such
	that one of  $M_G$ or $M_H$ is empty and   $\matchno_k(G\star H) = |\fbast(M_G,M_H)|$.
	If one of the sets 	$M_G$ or $M_H$ is non-empty, then (M1.b),  (M2.b) or (M3) is satisfied.
	By 	Prop.\ \ref{prop:matching-cast}, $\fcast(M_G,M_H)$ is a $k$-matching.
  By	Prop.\ \ref{prop:maxexprequal} and Prop.\ \ref{prop:size-cast},
  we obtain  $\matchno_k(G\star H) = |\fbast(M_G,M_H)| = |\fcast(M_G,M_H)|$.
	If both  	$M_G$ or $M_H$ are empty $k$-matchings, then $\fbast(M_G,M_H)=\fcast(M_G,M_H)=\emptyset$.
	In this case, $\matchno_k(G\star H) = |\fbast(M_G,M_H)| = 0 = |\fcast(M_G,M_H)|$.
	  Hence, in all cases,  $\matchno_k(G\star H)$ is   $\circledast$-well-behaved  w.r.t.\ $G$ and $H$.
\end{proof}


\subsection{$\boldsymbol{\ast}$-well-behaved}
\label{subsec:ast-wb}

Again, we will make frequent use of the condition (M1), (M2), (M3) and (M4) as specified in Proposition\ \ref{prop:matching-cast}.

\begin{definition}
Let $\star\in\{\sprod,\lprod,\dprod\}$. The $k$-matching number $\matchno_k(G\star
H)$ of $G\star H$ is \emph{$\ast$-well-behaved w.r.t.\ $G$ and $H$}
if  \[\matchno_k(G\star H) = |\fast(M_{G}, M_{H})|\]
for some $k$-matching $\fast(M_{G}, M_{H})$ of $G\star H$.
In this case, the $k$-matching $\fast(M_{G}, M_{H})$
		is also called \emph{$\ast$-well-behaved}.
\end{definition}

\begin{proposition}\label{prop:size-ast}
 Let $\star\in \{\sprod,\lprod,\dprod\}$ and
  $G$ and $H$ be two graphs of order $n_G$ and $n_H$, respectively. Moreover, let
 $M_G\subseteq E_G$ and $M_H\subseteq E_H$ and denote the number of $M_G$-unmatched vertices in $G$
 and $M_H$-unmatched vertices in $H$ by $u_G$ and $u_H$, respectively.
			If 	$\fast(M_G,M_H)$  is a $k$-matching of $G\star H$, then
			 \[|\fast(M_G,M_H)| = \frac{k}{2} \left(n_G-u_G\right)\left(n_H-u_H \right). \]
\end{proposition}
\begin{proof}
	 Let $G$ and $H$ be of order $n_G$ and $n_H$, respectively and $M_G\subseteq E_G$
	 and $M_H\subseteq E_H$. Assume that $\fast(M_G,M_H)$ is a $k$-matching of $G\star
	 H$. If it is empty, then at least one of $M_G$ and $M_H$ is empty, so that $n_G=u_G$ (resp. $n_H=u_H$).
	 Then the statement holds trivially, since 
	 $$|\fast(M_G,M_H)| = 0 = \frac{k}{2} \left(n_G-u_G\right)\left(n_H-u_H \right).$$
	 Hence, suppose that neither $M_G$ nor $M_H$ is empty.
	 By Prop.\ \ref{prop:matching-ast}, $M_G$ must be a $k_G$-matching of $G$ and
	 $M_H$ be a $k_H$-matching of $H$, where $k_G\cdot k_H=k$. By definition,
	 $|\fast(M_G,M_H)|=2|M_G||M_H|$ since for each pair $\{g,g'\}\in M_G$ and
	 $\{h,h'\}\in M_H$ the two edges $\{(g,h),(g',h')\}$ and $\{(g,h'),(g',h)\}$ are
	 contained in $\fast(M_G,M_H)$. By Lemma \ref{lem:muFormula}, we thus obtain
	 \[|\fast(M_G,M_H)|=2\frac{k_G(n_G-u_G)}{2}\frac{k_H(n_H-u_H)}{2} = \frac{k}{2}
	 \left(n_G-u_G\right)\left(n_H-u_H \right),\] where we used $k_G\cdot k_H=k$ in the
	 latter equation.
\end{proof}

\begin{lemma}
	Let $\star\in\{\sprod,\lprod\}$ and let $k$ be a prime number.
	If $\matchno_k(G\star H)$ is $\ast$-well-behaved
	w.r.t.\ $G$ and $H$, then $\matchno_k(G\star H)$ is
	$\circledast$-well-behaved w.r.t.\ $G$ and $H$ and $\MU_G=\emptyset$ or $\MU_H=\emptyset$.
\end{lemma}
\begin{proof}
	Assume that $\matchno_k(G\star H)$ is $\ast$-well-behaved w.r.t.\ $G$ and $H$.
	The statement is vacuously true if $\matchno_k(G\star H) = 0$.
	Assume now that $\matchno_k(G\star H) \neq 0$.
  By	definition, there is a non-empty $k$-matching $\fast(M_{G}, M_{H})$ of $G\star H$.
	This together with Prop.\ \ref{prop:matching-ast}  and the fact  that $\matchno_k(G\star H)\neq 0$ and  $k$ is prime
	implies that there is a non-empty $k$-matching $M_G$ of $G$ and a non-empty $1$-matching of $M_H$
	or \emph{vice versa}. W.l.o.g.\ assume that $M_G$ is a non-empty $k$-matching of $G$ and
	that $M_G$ is a non-empty $1$-matching of $G$.
	If $M_G$ is a perfect $k$-matching  of $G$, then (M1.a) is satisfied
	and  $\fcast(M_{G}, M_{H})$ is a perfect $k$-matching of $G\star H$ (cf.\ Prop.\ \ref{prop:size-cast}).
	and so	$\matchno_k(G\star H)\neq 0$ is 	$\circledast$-well-behaved and $\MU_G=\emptyset$.

	Now assume that $M_G$ is not a perfect $k$-matching   of $G$ and define $M'_H=\emptyset$. 
	Hence, $M_G$ and $M'_H$ satisfy (M1.b)
	and thus, $\fcast(M_{G}, M'_{H})$ is a $k$-matching of $G\star H$ (cf.\ Prop.\ \ref{prop:matching-cast}).
	We show  that $|\fcast(M_{G}, M'_{H})| = |\fast(M_{G}, M_{H})|$.
	Since  $M'_H=\emptyset$ we have $\fcast(M_{G}, M'_{H}) = \MU_H$ and
	thus,  $|\fcast(M_{G}, M'_{H})| = n_H|M_G|$.
	Now consider  $\fast(M_{G}, M_{H})$.
	Since $M_H$ is a 1-matching, 
	the edges in $M_H$ are pairwise vertex disjoint. Hence,
	$|M_H| = \frac{1}{2} |W_H|$ with $W_H = \cup_{e\in M_H}e \subseteq V_H$.
	Thus, we can write  $|\fast(M_{G}, M_{H})| = 2|M_G||M_H| = |M_G||W_H|$.
	Clearly, $|\fcast(M_{G}, M'_{H})| = n_H|M_G|\geq |M_G||W_H| = |\fast(M_{G}, M_{H})|$.
	However, since both  $\fcast(M_{G}, M'_{H})$ and $\fast(M_{G}, M_{H})$ are  $k$-matchings
	and since $\fast(M_{G}, M_{H}) = \matchno_k(G\star H)$, it must hold that
	$|\fcast(M_{G}, M'_{H})|=\matchno_k(G\star H)$ and thus,
	 $\matchno_k(G\star H)\neq 0$ is
	$\circledast$-well-behaved w.r.t.\ $G$ and $H$ and $\MU_G=\emptyset$.
\end{proof}


We finally provide a characterization of  $\ast$-well-behaved $k$-matching numbers 
which is  based on the structural properties of the factors.

\begin{theorem}
	Let $\star\in\{\sprod,\lprod,\dprod\}$.
	The following  statements are equivalent for all graphs $G$ and $H$ and every integer $k\geq 1$.
	\begin{enumerate}[itemsep=0.02em]
		\item $\matchno_k(G\star H)$ is $\ast$-well-behaved w.r.t.\ $G$ and $H$
		\item There is some $k_G$-matching $M_G$ of $G$ and	some $k_H$-matching $M_H$ of $H$
					such that $\unmatchno_k(G\star H) = n_Gu_H + n_Hu_G - u_Gu_H$ 
					with $u_G$ and $u_H$ being the $M_G$- and $M_H$-unmatched vertices in $G$ and $H$, respectively.
					In this case, $M_G$ and $M_H$
					are maximum matchings of $G$ respectively $H$.
		\item $\matchno_k(G\star H) = 2 \matchno_{k_G}(G) \matchno_{k_H}(H)$ for some positive integers $k_G$ and $k_H$
		      satisfying $k=k_Gk_H$.
	\end{enumerate}
\end{theorem}
\begin{proof}
	 We show first that (1) implies (3).
	Assume that  $\matchno_k(G\star H)$ is $\ast$-well-behaved w.r.t.\ $G$ and $H$. 
	Hence, there is a $k$-matching $\fast(M_G,M_H)$ of $G\star H$ such that 
	 $\matchno_k(G\star H)= |\fast(M_G,M_H)| = 2|M_G||M_H|$ where
	 $M_G$ is a $k_G$-matching of $G$ and $M_H$ is a $k_H$-matching of $H$ with $k_G\cdot k_H=k$
	 (cf.\ Prop.\	\ref{prop:matching-ast} where we choose $M_H=M_G=\emptyset$ in case
	  $\matchno_k(G\star H)=0$ and thus, $\fast(M_G,M_H)=\emptyset$). 
	  Clearly, $\matchno_k(G\star H) = 2|M_G||M_H|\leq 2\matchno_{k_G}(G) \matchno_{k_H}(H)$.
	 For a maximum  maximum $k_G$-matching $M^*_G$ of $G$ and a maximum $k_H$-matching $M^*_H$ of $H$ we have
	 $\matchno_k(G\star H) \geq |\fast(M^*_G,M^*_H)| = 2|M^*_G||M^*_H|= 2\matchno_{k_G}(G) \matchno_{k_H}(H)$. 
	 In summary, $\matchno_k(G\star H) =  2\matchno_{k_G}(G) \matchno_{k_H}(H)$.
	 
	 We now show that (3) implies (2). 
	 Assume that $\matchno_k(G\star H) = 2 \matchno_{k_G}(G) \matchno_{k_H}(H)$ for some positive integers $k_G$ and $k_H$
   with $k=k_Gk_H$. Let $M_G$ be a maximum $k_G$-matching of $G$ and $M_H$ be a maximum $k_H$-matching of $H$.
	 Lemma \ref{lem:muFormula} implies that
	 $\matchno_k(G\star H) =  \frac{k}{2}\left(n_{G\star H}-\unmatchno_k(G\star H)\right)$. 
	 By definition, $\matchno_k(G\star H) = 2|M_H||M_G|$. 
	 This together with Lemma \ref{lem:muFormula} and $k=k_Gk_H$ implies that
	 $\matchno_k(G\star H) = 2\frac{k_G}{2}\left(n_{G}-\unmatchno_{k_G}(G)\right)\frac{k_H}{2}\left(n_{H}-\unmatchno_{k_H}(H)\right)
	  = 	\frac{k}{2} \left(n_Gn_H-n_G\unmatchno_{k_H}(H)-n_H\unmatchno_{k_G}(G)+\unmatchno_{k_G}(G)\unmatchno_{k_G}(H) \right)$. 
	  Since $n_{G\star H} = n_Gn_H$, we thus conclude that 
	  $\unmatchno_k(G\star H) = n_G\unmatchno_{k_H}(H)+n_H\unmatchno_{k_G}(G)-\unmatchno_{k_G}(G)\unmatchno_{k_G}(H)$. 
	  
	We finally show that (2) implies (1). If $M_G=\emptyset$ or $M_H=\emptyset$ we obtain
	the empty $k$-matching $\fast(M_G,M_H)$. In any other case,
	Prop.\	\ref{prop:matching-ast} implies that $\fast(M_G,M_H)$ is a non-empty 
	$k$-matching of $G\star H$.
	We must verify that 	 $\matchno_k(G\star H)= |\fast(M_G,M_H)|$ whether or not $\fast(M_G,M_H)$ is empty.
	As argued above, 
	$\matchno_k(G\star H)=	\frac{k}{2}\left(n_{G}n_{H}-\unmatchno_k(G\star H)\right)$
	and since  $\unmatchno_k(G\star H) = n_Gu_H + n_Hu_G - u_Gu_H$, we obtain
	$\matchno_k(G\star H)=	\frac{k}{2}\left(n_{G}n_{H}- n_Gu_H - n_Hu_G + u_Gu_H\right) =
	\frac{k}{2} \left(n_G-u_G\right)\left(n_H-u_H \right)$. This together with
	 Prop.\	\ref{prop:size-ast} implies  $\matchno_k(G\star H)= |\fast(M_G,M_H)|$
	 and, therefore, $\matchno_k(G\star H)$ is $\ast$-well-behaved w.r.t.\ $G$ and $H$. 
\end{proof}


\section{Weak-homomorphism preserving matchings in products}
\label{sec:wh}

We focused here on constructions of $k$-matchings in products that
are determined by some unambiguous predefined rule based on the matchings in the factors.
These rules have not been chosen arbitrarily and satisfy, in particular, the
following natural property of being weak-homomorphism preserving.

\begin{definition}
Let $\star\in \{\dprod, \cprod, \sprod, \lprod\}$. Moreover,
let $G$ and $H$ be two graphs with $M_G\subseteq E_G$ and $M_H\subseteq E_H$.
Let $M$ be any subset  of $E_{G\star H}$ whose construction is
based on $M_G$ and $M_H$ in an arbitrary way.
This constructed set $M$ is
\emph{weak-homomorphism preserving} 
w.r.t.\ $M_G$ and $M_H$
if for all $e\in M$ it holds that
$p_G(e)\in M_G\cup V_G$ and $p_H(e)\in M_H\cup V_H$.

The set $\mathcal{W}_k(G\star H, M_G,M_H)$ denotes the collection of all
$k$-matchings of $G\star H$ that are weak-homomorphism preserving w.r.t.\ $M_G$ and $M_H$.
\end{definition}

In other words, a $k$-matching $M$ that is constructed from information provided by $M_G$ and $M_H$
is weak-homomorphism preserving
if the constructed matched edges of $G\star H$ in $M$  are never mapped
to unmatched edges in the factors $G$ and $H$.
Note, $\mathcal{W}_k(G\star H, M_G,M_H)\neq \emptyset$ for every integer $k\geq 1$,
every product
$\star\in \{\dprod, \cprod, \sprod, \lprod\}$ and all sets $M_G$ and $M_H$,
since $M=\emptyset$ is trivially a weak-homomorphism preserving $k$-matching
for all $G\star H$ and thus, $\emptyset \in \mathcal{W}_k(G\star H, M_G,M_H)$.

Clearly, if $M_G=E_G$ and $M_H=E_H$, then $M\in \mathcal{W}_k(G\star H, M_G,M_H)$
for all positive integers $k$, all $M\subseteq E_{G\star H}$ and $\star\in \{\cprod,\dprod, \sprod,\lprod\}$.
Moreover, it is easy to verify that, by definition, 
$\fbast(M_G,M_H) \in \mathcal{W}_k(G\star H, M_G,M_H)$
for the products $\star\in \{\cprod,\sprod,\lprod\}$,
$\fcast(M_G,M_H) \in \mathcal{W}_k(G\star H, M_G,M_H)$ for the products $\star\in \{\sprod,\lprod\}$,
and $\fast(M_G,M_H)\in \mathcal{W}_k(G\star H, M_G,M_H)$ for the products $\star\in \{\dprod,\sprod,\lprod\}$,
provided that the respective edge-set is a $k$-matching of $G\star H$.
An illustration of $\fcast(M_G,M_H)$ is provided in  Fig.\ \ref{fig:fcast-exmpl}.
Further constructions that do not correspond to either of the sets
$\fbast(M_G,M_H)$, $\fcast(M_G,M_H)$ and $\fast(M_G,M_H)$ but
are weak-homomorphism preserving w.r.t.\ $M_G$ and $M_H$ are shown in Fig.\ \ref{fig:art1} and \ref{fig:strange4M-2}.
An example that is not weak-homomorphism preserving is shown in
Fig.~\ref{fig:inducedMatchings}(right) as well as in Fig.\ \ref{fig:notMaximal}
(the perfect $2$-matching in $G\cprod H$ highlighted the dashed, dashed-dotted and bold lined edges).

\begin{figure}[t]
	\begin{center}
		\includegraphics[width=0.6\textwidth]{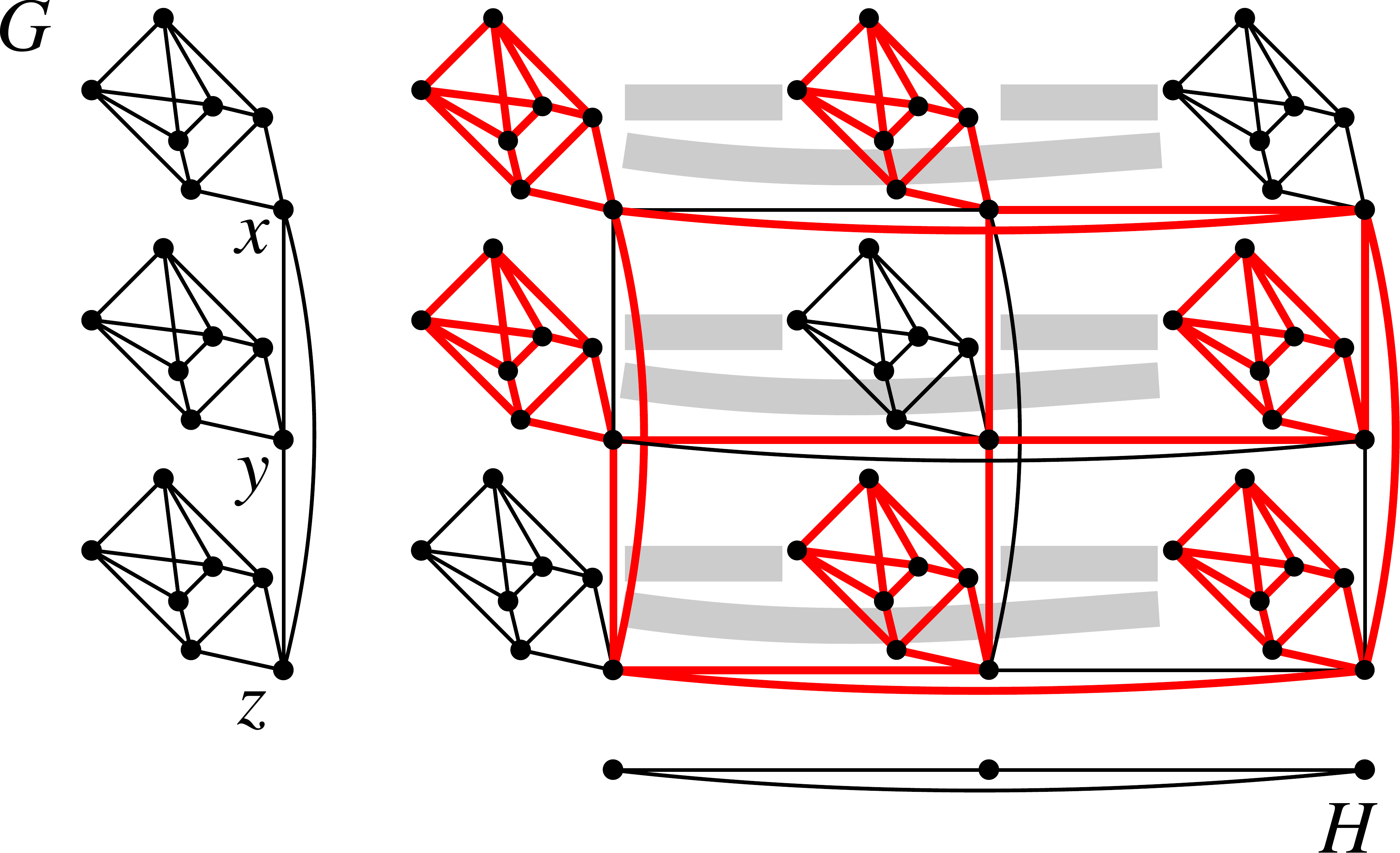}
	\end{center}
	\caption{Shown is the Cartesian product $G\cprod H$ where the copies of $H$ are
	         just sketched (gray lines), except the layers $\lefts{x}{H}$,
	         $\lefts{y}{H}$ and $\lefts{z}{H}$. In this example, $M_G=E_G$ is a perfect
	         $4$-matching of $G$ and $M_H=E_H$ a perfect 2-matching of $H$. The set
	         $M\subseteq E_{G\cprod H}$ consisting of the red-bold edges is a
	         weak-homomorphism preserving $4$-matching of $G\cprod H$ and thus, $M\in
	         \mathcal{W}_4(G\star H,M_G,M_H)$. However, $M$ is not a maximum-sized set among the elements in
	         $\mathcal{W}_4(G\star H,M_G,M_H)$ and $\matchno_4(G\star H)$ is not well-behaved
	         for this particular construction $M$,
	         since $\fbast(M_G,M_H)\in\mathcal{W}_4(G\star H,M_G,M_H)$
	         (which consists of edges the three copies of $G$) yields a perfect 4-matching.
	         Note, $\fbast(M_H,M_G)$ would consist of the edges of all copies of $H$
	         and would yield a perfect 2-matching. In particular, $\fbast(M_H,M_G)\in\mathcal{W}_2(G\star H,M_G,M_H)$.
	         }
		\label{fig:art1}
\end{figure}

\begin{figure}[ht]
	\begin{center}
		\includegraphics[width=0.75\textwidth]{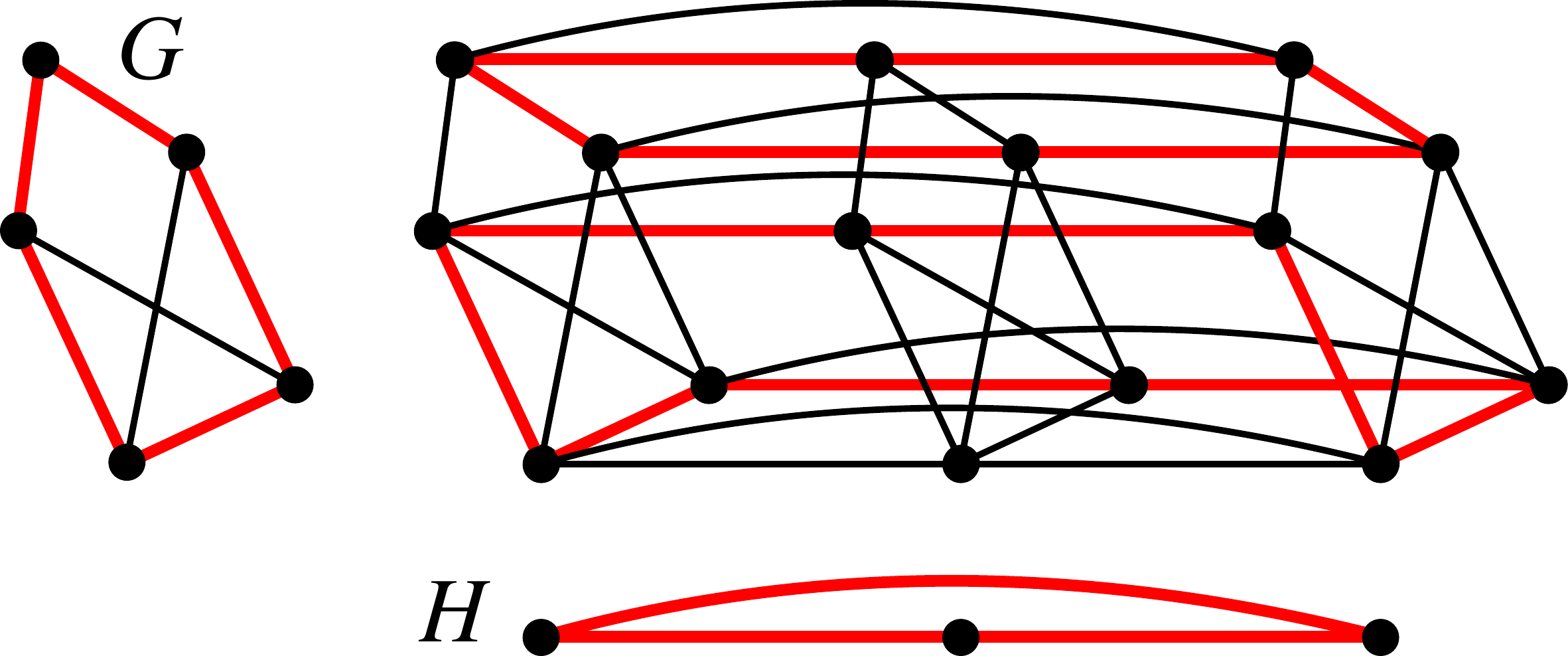}
	\end{center}
	\caption{Shown is the Cartesian product $G\cprod H$ with near-perfect 2-matching $M$ where the two factors 
							 $G$ and $H$ have a perfect 2-matching $M_G$ and $M_H$, respectively. The respective 2-matchings are 
							 highlighted by bold-red edges.
							 It is easy to verify that  $M\in \mathcal{W}_2(G\cprod H, M_G,M_H)$. Moreover, 
							 $\fbast(M_G,M_H)$ consists of the copies of edges in $M_G$ along the $G$-layers and 
							 $\fbast(M_H,M_G)$ consists of the copies of edges in $M_H$ along the $H$-layers.  
							 Since $\fbast(M_G,M_H),\fbast(M_H,M_G)\in  \mathcal{W}_2(G\cprod H, M_G,M_H)$ and 
							$|\fbast(M_G,M_H)| = |\fbast(M_H,M_G)|=	15$ but $|M|=14$, Thm.\ \ref{thm:max-Wk-fbast}
							 implies that $M$ is not a maximum matching. 
	}
		\label{fig:strange4M-2}
\end{figure}

In what follows, we show that $k$-matchings $\fbast(M_G,M_H)$,  $\fcast(M_G,M_H)$ and $\fast(M_G,M_H)$
are always maximum-sized elements in $\mathcal{W}_k(G\star H, M_G,M_H)$ for particular specified products $\star$.
To this end, we need first the following result.

\begin{lemma}\label{lem:gh-unmatch}
	Let $G$ and $H$ be graphs and $M_G\subseteq E_G$ and $M_H\subseteq E_H$.
	If  $g\in V_G$ is $M_G$-unmatched in $G$ and $h\in V_H$ is $M_H$-unmatched in $H$
	then $(g,h) \in V_{G\star H}$ is $M$-unmatched for all $M\in \mathcal{W}_k(G\star H, M_G,M_H)$ with  $\star\in \{\dprod, \cprod, \sprod, \lprod\}$.
	If $(g,h) \in V_{G\star H}$ is $M$-unmatched for all $M\in \mathcal{W}_k(G\star H, M_G,M_H)$ with  $\star\in \{\cprod, \sprod, \lprod\}$, 
	then $g\in V_G$ is $M_G$-unmatched in $G$ and $h\in V_H$ is $M_H$-unmatched in $H$.
\end{lemma}
\begin{proof}
	Suppose that $g\in V_G$ is $M_G$-unmatched in $G$ and $h\in V_H$ is $M_H$-unmatched in $H$.
	Assume, for contradiction, that $(g,h)$ is $M$-matched for some $M\in \mathcal{W}_k(G\star H, M_G,M_H)$ with  $\star\in \{\dprod, \cprod, \sprod, \lprod\}$.
	Thus, vertex $(g,h)$ is incident to some edge $e=\{(g,h),(g',h')\} \in M$ (note,  $g\neq g'$ or $h\neq h'$).
	In this case, however, the projection of this edge $e$ onto the factors coincides with the edge $\{g,g'\}\in E_G\setminus M_G$ in $G$
	or $\{h,h'\}\in E_H\setminus M_H$ in $H$.
	Hence, we have mapped a matched edge $e$ in $G\star H$ to an unmatched edge in one of the factors;
	a contradiction to the property of being weak-homomorphism preserving. Therefore, $(g,h)$ must be $M$-unmatched for all $M\in \mathcal{W}_k(G\star H, M_G,M_H)$.

	Suppose now that $(g,h)$  is $M$-unmatched for all $M\in \mathcal{W}_k(G\star H, M_G,M_H)$ with  $\star\in \{\cprod, \sprod, \lprod\}$.
	If $g$ would
	be $M_G$-matched in $G$, then there is an edge $\{g,g'\}\in M_G$. By definition
	$\{(g,h),(g',h)\}\in \MB_G \in \mathcal{W}_k(G\star H, M_G,M_H)$ and thus, $(g,h)$ is not
	$\MB_G$-unmatched in $G\star H$; a contradiction. Assume, for contradiction,
	that $h$ is $M_H$-matched in $H$ and hence, there is an edge $\{h,h'\}\in M_H$. By
	the latter argument, $g$ must $M_G$-unmatched in $G$ and thus, $g\in U_G$. By
	definition, $\{(g,h),(g,h')\}\in \MU_G \in \mathcal{W}_k(G\star H, M_G,M_H)$ and therefore,
	$(g,h)$ is not $\MU_G$-unmatched in $G\star H$; a contradiction.
	Therefore, $g$ is $M_G$-unmatched in $G$ and $h$ is $M_H$-unmatched in $H$.
\end{proof}

Intriguingly, the latter results allow us to determine the size of $k$-matchings
that are maximum among all weak-homomorphism preserving $k$-matchings provided that some
knowledge of the underlying sets $M_G$ and $M_H$ is available, as shown in the following
theorems.

\begin{theorem}\label{thm:max-Wk-fbast}
	Let $\star\in \{\cprod, \sprod, \lprod\}$.
	Let $G$ and $H$ be graphs, $M_G\subseteq E_G$ and $M_H\subseteq E_H$ and $u_G$ be the number of $M_G$-unmatched vertices in $G$
and  $u_H$ be the number of $M_H$-unmatched vertices in $H$. 	

 If $\fbast(M_G,M_H)$ is a $k$-matching of $G\star H$,
	then $\fbast(M_G,M_H)\in \mathcal{W}_k(G\star H, M_G,M_H)$ and
	$|\fbast(M_G,M_H)|\geq |M|$ for all $M\in \mathcal{W}_k(G\star H, M_G,M_H)$.

In particular, if $M_G$ is a perfect $k$-matching of $G$ or
 $M_H$ is a perfect $k$-matching of $H$ or $M_G$ and $M_H$ are $k$-matchings of $G$ and $H$, respectively,
then \[|M^*| =  \frac{k}{2}\left (n_Gn_H-u_Gu_H \right)\]
for every every maximum-sized element $M^*\in \mathcal{W}_k(G\star H, M_G,M_H)$.
\end{theorem}
\begin{proof}
	Let $M\in \mathcal{W}_k(G\star H, M_G,M_H)$ and $\star\in \{\cprod, \sprod, \lprod\}$.
	By definition of $\fbast(M_G,M_H)$, we have
	$\fbast(M_G,M_H)\in \mathcal{W}_k(G\star H, M_G,M_H)$.
	By Lemma \ref{lem:fbast-gh-unmatch},
	if $(g,h)$ is an $\fbast(M_G,M_H)$-unmatched vertex in $G\star H$
	then $g$ is $M_G$-unmatched in $G$ and $h$ is $M_H$-unmatched in $H$.
	This together with Lemma \ref{lem:gh-unmatch} implies that
	all $\fbast(M_G,M_H)$-unmatched vertices
	in $G\star H$ are also $M$-unmatched in $G\star H$.
	Therefore, $M$ can only contain edges between $\fbast(M_G,M_H)$-matched vertices.
	Hence, for the number $u'$ of $M$-unmatched vertices and
	$u$ of $\fbast(M_G,M_H)$-unmatched vertices it must hold $u'\geq u$.
	By Lemma \ref{lem:muFormula},
	$|\fbast(M_G,M_H)|= \frac{k}{2} (n_{G\star H}-u) \geq  \frac{k}{2} (n_{G\star H}-u') =|M|$.
	
	Now, let  $M^*$ be any maximum-sized element in $\mathcal{W}_k(G\star H, M_G,M_H)$,
	and assume that 
	$M_G$ is a perfect $k$-matching of $G$ or
 $M_H$ is a perfect $k$-matching of $H$ or $M_G$ and $M_H$ are $k$-matchings of $G$ and $H$, respectively. 
  Hence, Prop.\ \ref{prop:matching-fac} implies that at least one of $\fbast(M_G,M_H)$ and
$\fbast(M_H,M_G)$ is a $k$-matching of $G\star H$. Thus, at least one satisfies $\fbast(M_G,M_H)\in
\mathcal{W}_k(G\star H, M_G,M_H)$ or $\fbast(M_H,M_G)\in \mathcal{W}_k(G\star H,
M_G,M_H)$. By the preceding arguments we have $|M^*| = |\fbast(M_G,M_H)|$ or
$|M^*| = |\fbast(M_H,M_G)|$.  This together with
Prop.\ \ref{prop:maxexprequal} implies that $|M^*| = \frac{k}{2}\left (n_Gn_H-u_Gu_H\right)$
\end{proof}

\begin{theorem}\label{thm:max-Wk-fcast}
	Let $\star\in \{\sprod, \lprod\}$.
	Let $G$ and $H$ be graphs, $M_G\subseteq E_G$ and $M_H\subseteq E_H$ and
	let $u_G$ be the number of $M_G$-unmatched vertices in $G$
	and  $u_H$ be the number of $M_H$-unmatched vertices in $H$. 	
	
	If $\fcast(M_G,M_H)$ is a $k$-matching of $G\star H$,
	then $\fcast(M_G,M_H)\in \mathcal{W}_k(G\star H, M_G,M_H)$ and
	$|\fcast(M_G,M_H)|\geq |M|$ for all $M\in \mathcal{W}_k(G\star H, M_G,M_H)$.
	
	In particular, if $M_G$ and $M_H$ satisfy one of the Conditions (M1)-(M4)
	in Prop.\ \ref{prop:matching-cast},  then 
	\[|M^*| =  \frac{k}{2}\left (n_Gn_H-u_Gu_H \right)\]
  for every every maximum-sized element $M^*\in \mathcal{W}_k(G\star H, M_G,M_H)$.
\end{theorem}
\begin{proof}
	By definition of $\fcast(M_G,M_H)$, we have
	$\fcast(M_G,M_H)\in \mathcal{W}_k(G\star H, M_G,M_H)$ for $\star\in \{\sprod, \lprod\}$.
	
	Lemma\ \ref{lem:fcast-gh-unmatch} implies that for every $\fcast(M_G,M_H)$-unmatched vertex $(g,h)$
	in $G\star H$ the vertex $g$ is $M_G$-unmatched in $G$ and $h$ is $M_H$-unmatched in $H$.
	This together with Lemma \ref{lem:gh-unmatch} implies that
	all $\fcast(M_G,M_H)$-unmatched vertices
	in $G\star H$ are also $M$-unmatched in $G\star H$ for all $M\in \mathcal{W}_k(G\star H, M_G,M_H)$.

	Thus, $M$ can only contain edges between $\fcast(M_G,M_H)$-matched vertices.
	Hence, for the number $u'$ of $M$-unmatched vertices and
	$u$ of $\fcast(M_G,M_H)$-unmatched vertices it must hold $u'\geq u$.
	By Lemma \ref{lem:muFormula},
	$|\fcast(M_G,M_H)| = \frac{k}{2} (n_{G\star H}-u) \geq  \frac{k}{2} (n_{G\star H}-u') =|M|$.

	Now, let  $M^*$ be any maximum-sized element in $\mathcal{W}_k(G\star H, M_G,M_H)$ and 
	assume that $M_G$ and $M_H$ satisfy one of the Conditions (M1)-(M4).
	By Prop.\ \ref{prop:matching-cast}, $\fcast(M_G,M_H)$ is a $k$-matching of $G\star H$.
	Thus, $\fcast(M_G,M_H)\in \mathcal{W}_k(G\star H, M_G,M_H)$ and $|M^*| = |\fcast(M_G,M_H)|$.
	By Prop.\ \ref{prop:size-cast}, $|M^*| = \frac{k}{2}\left (n_Gn_H-u_Gu_H \right)$.
\end{proof}

A result that we already observed in Prop.\ \ref{prop:maxexprequal}
and Prop.\ \ref{prop:size-cast}, follows now directly from Theorem
\ref{thm:max-Wk-fbast} and \ref{thm:max-Wk-fcast}.
\begin{corollary}
	Let $\star\in \{\sprod, \lprod\}$ and
	let $G$ and $H$ be graphs and $M_G\subseteq E_G$ and $M_H\subseteq E_H$. 	If $\fbast(M_G,M_H)$
	and $\fcast(M_G,M_H)$ are $k$-matchings then $|\fbast(M_G,M_H)| = |\fcast(M_G,M_H)|$.
\end{corollary}
\begin{proof}
If $\fbast(M_G,M_H)$	and $\fcast(M_G,M_H)$ are $k$-matchings,
then they are both contained in $ \mathcal{W}_k(G\star H, M_G,M_H)$.
By Thm.\ \ref{thm:max-Wk-fbast} and \ref{thm:max-Wk-fcast}, both
$\fbast(M_G,M_H)$	and $\fcast(M_G,M_H)$ are maximum-sized elements in
$\mathcal{W}_k(G\star H, M_G,M_H)$ and thus, must have the same cardinality.
\end{proof}

For the direct product, we provide
\begin{theorem}\label{thm:max-Wk-fast}
	Let $G$ and $H$ be graphs and $M_G\subseteq E_G$ and $M_H\subseteq E_H$
	and let $u_G$ be the number of $M_G$-unmatched vertices in $G$
	and  $u_H$ be the number of $M_H$-unmatched vertices in $H$. 	

	If $\fast(M_G,M_H)$ is a $k$-matching,
	then $\fast(M_G,M_H)\in \mathcal{W}_k(G\dprod H, M_G,M_H)$. 
	Moreover, $M\subseteq \fast(M_G,M_H)$ and thus,
	$|\fast(M_G,M_H)|\geq |M|$ for all $M\in \mathcal{W}_k(G\dprod H, M_G,M_H)$.
	
	In particular, if $M_G$ is a $k_G$-matching of $G$ and $M_H$ is a 
	$k_H$-matching of $H$ such that $k_G\cdot k_H=k$
	then
			 \[|M^*| = \frac{k}{2} \left(n_G-u_G\right)\left(n_H-u_H \right) \]
  for every every maximum-sized element $M^*\in \mathcal{W}_k(G\dprod H, M_G,M_H)$.
\end{theorem}
\begin{proof}
	By definition of $\fast(M_G,M_H)$, we have
	$\fast(M_G,M_H)\in \mathcal{W}_k(G\dprod H, M_G,M_H)$.
	It suffices show that $M\subseteq \fast(M_G,M_H)$ for all 
	$M\in \mathcal{W}_k(G\dprod H, M_G,M_H)$.
	Let $\{(g,h),(g',h')\} \in M \in \mathcal{W}_k(G\dprod H, M_G,M_H)$.
	Since $M\subseteq E_{G\dprod H}$ it must hold that
	$\{g,g'\}\in E_G$ and $\{h,h'\}\in E_H$. 
	Since $M$ is weak-homomorphism preserving w.r.t.\ $M_G$ and $M_H$, it follows that
	$\{g,g'\}\in M_G$ and $\{h,h'\}\in M_H$. By definition, 
  $\{(g,h),(g',h')\}\in \fast(M_G,M_H)$ and thus,  $M\subseteq \fast(M_G,M_H)$ 
  which implies 	$|\fast(M_G,M_H)|\geq |M|$ for all $M\in \mathcal{W}_k(G\dprod H, M_G,M_H)$.
  
	Now, let  $M^*$ be any maximum-sized element in $\mathcal{W}_k(G\dprod H, M_G,M_H)$ and 
	assume that $M_G$ is a $k_G$-matching of $G$ and $M_H$ is a 
	$k_H$-matching of $H$ such that $k_G\cdot k_H=k$. 
	 If $M_G=\emptyset$ or $M_H=\emptyset$ we obtain
	the empty $k$-matching $\fast(M_G,M_H)$. In any other case,
  Prop.\ \ref{prop:matching-ast} implies that $\fast(M_G,M_H)$ is a $k$-matching of $G\dprod H$. 
  Thus, $\fast(M_G,M_H)\in \mathcal{W}_k(G\dprod H, M_G,M_H)$ and $|M^*| = |\fast(M_G,M_H)|$.
  By Prop.\ \ref{prop:size-ast}, $|M^*| = \frac{k}{2} \left(n_G-u_G\right)\left(n_H-u_H \right)$.
\end{proof}

We finally remark that maximum-sized elements in  $\mathcal{W}_k(G\star H, M_G,M_H)$ 
   are not necessarily maximum $k$-matchings of $G\star H$. 
   \begin{example}
		Consider the product $G\cprod H$ as shown in Fig.\ \ref{fig:inducedMatchings}. 
		The $1$-matchings $\fbast(M_G,M_H)$ and $\fbast(M_H,M_G)$ are, by Thm.\ \ref{thm:max-Wk-fbast}, 
		maximum-sized elements in $\mathcal{W}_k(G\cprod H, M_G,M_H)$. However, they are not
		perfect, although a perfect $1$-matching exists. Thus, $\fbast(M_G,M_H)$ and $\fbast(M_H,M_G)$
		are not maximum $1$-matching of   $G\cprod H$.
   \end{example}
   
   Moreover, although $\mathcal{W}_k(G\star H, M_G,M_H)$ may contain maximum $k$-matchings 
   of $G\star H$ there could be further maximum $k$-matchings   of $G\star H$ that do no 
   correspond to our constructions but are weak-homomorphism preserving. 
   \begin{example}
		Consider the product $G\cprod H$ as shown in Fig.\ \ref{fig:exmpl2-new}. 
		The $1$-matchings $\fbast(M_G,M_H), \fbast(M_H,M_G)\in \mathcal{W}_k(G\cprod H, M_G,M_H)$ are
		maximum $1$-matchings of $G\cprod H$. A further maximum $1$-matching of 
		$G\cprod H$ that is not contained in $\mathcal{W}_k(G\cprod H, M_G,M_H)$
		is shown in Fig.\ \ref{fig:exmpl2-new} (3rd from left).
   \end{example}

\section{Construction of $\boldsymbol{1}$-matchings}
\label{sec:1m}

Based on the previous findings, we collect here results that particularly hold for $1$-matchings.

Note, $k$-matchings $\fbast(M_G,M_H)$ or $\fbast(M_H,M_G)$ are in general not
maximal $k$-matchings of $G \star H$ (see Fig.\ \ref{fig:notMaximal}); a case that
cannot happen for $1$-matchings $\fbast(M_G,M_H)$ and $\fbast(M_H,M_G)$, provided
that $M_G$ and $M_H$ are maximal $1$-matchings.
\begin{lemma}\label{lem:max-1-fbast}
	Let $M_G$ and $M_H$ be maximal $1$-matchings of $G$ and $H$, respectively. Then,
	$\fbast(M_G,M_H)$ and $\fbast(M_H,M_G)$ are maximal $1$-matchings in $G\star H$,
	$\star\in\{\cprod,\sprod,\lprod\}$.
\end{lemma}
\begin{proof}
Let $\star\in\{\cprod,\sprod,\lprod\}$ and $M_G$ and $M_H$ be maximal $1$-matchings
of $G$ and $H$, respectively. By Cor.\ \ref{cor:matching-fac}, both $\fbast(M_G,M_H)$ and $\fbast(M_H,M_G)$ are
$1$-matchings of $G \star H$. 
	
Assume now, for contradiction, that $\fbast(M_G,M_H)$
is not a maximal $1$-matching of $G\star H$. Thus, there is a $1$-matching $M'$ of
$G\star H$ that satisfies $\fbast(M_G,M_H)\subsetneq M'$. Let $\{(g,h),(g',h')\}\in
M'\setminus (\fbast(M_G,M_H))$.
Hence, both vertices $(g,h)$ and $(g',h')$ are $\fbast(M_G,M_H)$-unmatched 
in $G\star H$. Since $\{(g,h),(g',h')\}\in E_{G\star H}$, 
it must hold $\{g,g'\}\in E_G$ or $\{h,h'\}\in E_H$ for every product $\star\in\{\cprod,\sprod,\lprod\}$.
Assume first that $\{g,g'\}\in E_G$. Lemma \ref{lem:fbast-gh-unmatch} implies that $g$ and $g'$
are $M_G$-unmatched in $G$. Thus, $M_G\cup \{\{g,g'\}\}$ is still a 1-matching
of $G$; a contradiction to the maximality of $M_G$. 
By similar arguments, the case $\{h,h'\}\in E_H$ yields a contradiction.
Therefore, $\fbast(M_G,M_H)$ must be a maximal $1$-matching of $G\star H$.

Similarly, $\fbast(M_H,M_G)$ must be a maximal $1$-matching of $G\star H$.
\end{proof}

In general, 1-matchings $\fast(M_G,M_H)$ are not maximal 1-matchings of $G\sprod H$ or $G\lprod H$, even if
both $M_G$ and $M_H$ are maximal 1-matchings of $G$ and $H$, respectively (see Fig. \ref{fig:exmpl-dprod} (right)). 
Nevertheless for the direct product, we obtain
\begin{lemma}
	Let $G$ and $H$ be graphs with non-empty edge set. 
	If $M_G$, resp., $M_H$ is a maximal $1$-matchings of $G$, resp., $H$, then
	$\fast(M_G,M_H)$ is a maximal $1$-matching of $G\dprod H$.
\end{lemma}
\begin{proof}
 Let $M_G$ and $M_H$ be a maximal $1$-matchings of $G$ and $H$, respectively. 
 Since both $G$ and $H$ contain edges, $M_G\neq \emptyset$ and $M_H\neq \emptyset$. 
 By Prop.\ \ref{prop:matching-ast}, $\fast(M_G,M_H)$ is a non-empty 
 $1$-matching of $G\dprod H$.
 
 Assume now, for contradiction, that $\fast(M_G,M_H)$
 is not a maximal $1$-matching of $G\dprod H$. Thus, there is a $1$-matching $M'$ of
 $G\dprod H$ that satisfies $\fbast(M_G,M_H)\subsetneq M'$. Let $\{(g,h),(g',h')\}\in
 M'\setminus (\fbast(M_G,M_H))$.
 Hence, both vertices $(g,h)$ and $(g',h')$ are $\fast(M_G,M_H)$-unmatched 
 in $G\dprod H$. Since $\{(g,h),(g',h')\}\in E_{G\dprod H}$, 
 we have $\{g, g'\}\in E_G$ and $\{h, h'\}\in E_H$.
 By construction of $\fast(M_G,M_H)$, we have $\{g, g'\}\notin M_G$ or $\{h, h'\}\notin M_H$.
 W.l.o.g.\ assume that $\{g, g'\}\notin M_G$. Since $M_G$ is a maximal $1$-matching of $G$, 
 $g$ or $g'$ must be $M_G$-matched in $G$. W.l.o.g.\ assume that 
 $g$ is $M_G$-matched in $G$. Hence, there is an edge $\{g,g''\}\in M_G$. 
 If  $\{h, h'\}\in M_H$, then $\{(g,h),(g'',h')\}\in \fast(M_G,M_H)$ and thus, 
 $(g,h)$ is not $\fast(M_G,M_H)$-unmatched; a contradiction. 
 Therefore,	$\{h, h'\}\notin M_H$. 
 Since $M_H$ is a maximal $1$-matching of $H$, 
 $h$ or $h'$ must be $M_H$-matched in $H$. W.l.o.g.\ assume that 
 $h$ is $M_H$-matched in $H$. Hence, there is an edge $\{h,h''\}\in M_H$.
 Now,  $\{g,g''\}\in M_G$ and $\{h,h''\}\in M_H$ imply that 
  $\{(g,h),(g'',h'')\}\in \fast(M_G,M_H)$ and thus, 
 $(g,h)$ is not $\fast(M_G,M_H)$-unmatched; a contradiction. 
 Therefore, $\fast(M_G,M_H)$ is  a maximal $1$-matching of $G\dprod H$.
\end{proof}

Similar as for $\fbast(M_G,M_H)$, a $k$-matching $\fcast(M_G,M_H)$ is in general not
a maximal $k$-matching of $G \star H$; a case that
cannot happen for $1$-matchings $\fcast(M_G,M_H)$, provided
that $M_G$ and $M_H$ are maximal $1$-matchings.

\begin{lemma}
	Let $M_G$ and $M_H$ be maximal $1$-matchings of $G$ and $H$, respectively. Then,
	$\fcast(M_G,M_H)$ is a maximal $1$-matching in $G\star H$,	$\star\in\{\sprod,\lprod\}$.
\end{lemma}
\begin{proof}
Let $\star\in\{\sprod,\lprod\}$ and $M_G$ and $M_H$ be maximal $1$-matchings
of $G$ and $H$, respectively. By Prop.\ \ref{prop:matching-cast}(M3), 
$\fcast(M_G,M_H)$ is a $1$-matchings of $G \star H$. 

Assume now, for contradiction, that $\fcast(M_G,M_H)$
is not a maximal $1$-matching of $G\star H$. Thus, there is a $1$-matching $M'$ of
$G\star H$ that satisfies $\fcast(M_G,M_H)\subsetneq M'$. Let $\{(g,h),(g',h')\}\in
M'\setminus (\fcast(M_G,M_H))$.
Hence, both vertices $(g,h)$ and $(g',h')$ are $\fcast(M_G,M_H)$-unmatched 
in $G\star H$. 
Since $\{(g,h),(g',h')\}\in E_{G\star H}$, 
it must hold $\{g,g'\}\in E_G$ or $\{h,h'\}\in E_H$ for every product $\star\in\{\sprod,\lprod\}$.
Assume first that $\{g,g'\}\in E_G$.
Lemma \ref{lem:fcast-gh-unmatch} implies that $g$ and $g'$
are $M_G$-unmatched in $G$. Thus, $M_G\cup \{\{g,g'\}\}$ is still a 1-matching
of $G$; a contradiction to the maximality of $M_G$. 
By similar arguments, the case $\{h,h'\}\in E_G$ yields a contradiction.
Therefore, $\fcast(M_G,M_H)$ must be a maximal $1$-matching of $G\star H$.
\end{proof}

\begin{figure}[t]
	\begin{center}
		\includegraphics[width=0.95\textwidth]{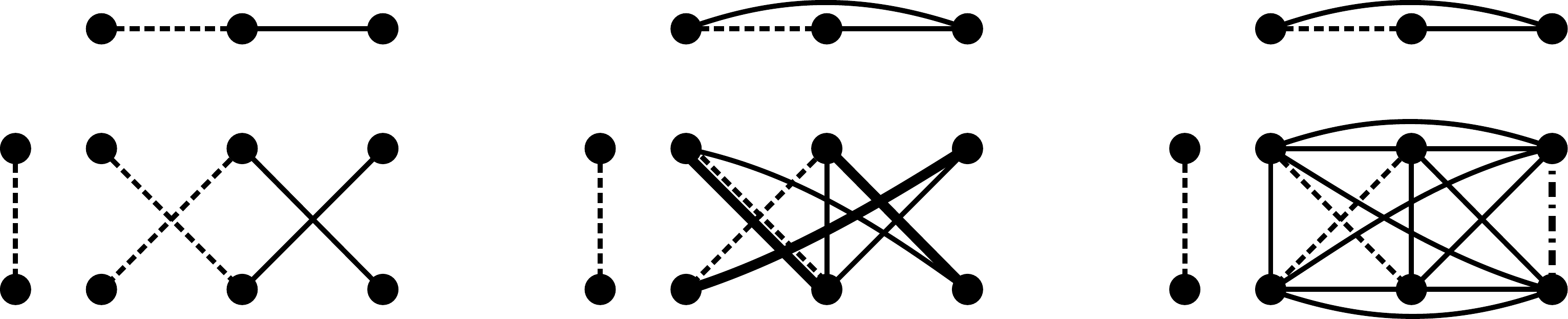}
	\end{center}
	\caption{\emph{Left panel:} Shown is the direct product 
	$K_2\dprod P_3$ together with the maximum $1$-matching $\fast(M_{K_2},M_{P_3})$
	highlighted by dashes edges.
	\emph{Middle panel:} Shown is the direct product 
	$K_2\dprod K_3$ together with the $1$-matching $\fast(M_{K_2},M_{K_3})$
	highlighted by dashed edges. In this example, $\fast(M_{K_2},M_{K_3})$
	is a maximal $1$-matching but not maximum. A maximum $1$-matching 
	of $K_2\dprod K_3$ is highlighted by bold-lined edges.
	\emph{Right panel:} Shown is the strong and lexicographic product 
	$K_2\sprod K_3\simeq K_2\lprod K_3$ together with the $1$-matching $\fast(M_{K_2},M_{K_3})$
	highlighted by dashed edges. The $1$-matching $\fcast(M_{K_2},M_{K_3})$ consists of
		$\fast(M_{K_2},M_{K_3})$ together with the bold dashed-dotted edge and is a maximum $1$-matching of the product.
	}
		\label{fig:exmpl-dprod}
\end{figure}

We characterize now $\ast$-well-behaved $1$-matchings for strong and lexicographic products.

\begin{proposition}\label{prop:1-m-fast}
	Let $\star\in\{\sprod,\lprod\}$. Then, $\matchno_1(G\star H)$ is $\ast$-well-behaved
	w.r.t.\ $G$ and $H$ if and only if $\matchno_1(G\star H)$ is
	$\circledast$-well-behaved w.r.t.\ $G$ and $H$ and $\MU_G=\MU_H=\emptyset$
	where $\MU_G$ and $\MU_H$ are defined in terms of a $1$-matching 
		$M_G$ of $G$ and	$M_H$   of $H$ that satisfy $\matchno_1(G\star H) = |\fcast(M_{G}, M_{H})|$.
\end{proposition}
\begin{proof}
	For the \emph{if}-direction observe that 
	if $\matchno_1(G\star H)$ is $\circledast$-well-behaved w.r.t.\ $G$ and $H$ and $\MU_G=\MU_H=\emptyset$,
	then there is a $1$-matching $\fcast(M_G,M_H)$ of $G\star H$
	such that  $\fcast(M_G,M_H) = \fast(M_G,M_H)$.  Thus, 
	$\matchno_1(G\star H) = |\fcast(M_G,M_H)| = |\fast(M_G,M_H)|$
	which implies that  $\matchno_1(G\star H)$ is $\ast$-well-behaved
	w.r.t.\ $G$ and $H$.

	For the \emph{only-if}-direction,
	assume that $\matchno_1(G\star H)$ is $\ast$-well-behaved w.r.t.\ $G$ and $H$. By
	definition, there is a $1$-matching $\fast(M_{G}, M_{H})$ of $G\star H$ such that $\matchno_1(G\star H)=|\fast(M_{G}, M_{H})|$.
	If it is empty, then $\matchno_1(G\star H)=0$, which means that $G\star H = (V_{G\star H}, \emptyset)$,
	otherwise we could at least have one edge in a maximum $1$-matching. Hence $\fcast(M_{G}, M_{H})=\emptyset$
	as well, and the statement is vacuously true.
	
	In the case $\fast(M_{G}, M_{H})$ is a non-empty $1$-matching, Prop.\ \ref{prop:matching-ast}
	implies that
	$M_G$ must be a non-empty $1$-matching of $G$ and
	$M_H$  a non-empty $1$-matching of $H$, such that $\matchno_k(G\star H) = |\fast(M_{G}, M_{H})|$.
	In this case, Prop.\ \ref{prop:matching-cast}(M3) is satisfied. Therefore, $\fcast(M_{G}, M_{H})$ is a
	$1$-matching of $G\star H$. Since
	$\fcast(M_{G}, M_{H}) = \fast(M_{G}, M_{H}) \cup \MU_G \cup \MU_H$ and
	$|\fast(M_{G}, M_{H})|= \matchno_k(G\star H) \geq |\fcast(M_{G}, M_{H})|$
	we obtain $\MU_G=\MU_H=\emptyset$ and
	$\fcast(M_{G}, M_{H}) = \fast(M_{G}, M_{H})$, which implies that
	$\matchno_k(G\star H) = |\fcast(M_{G}, M_{H})|$. Thus, $\matchno_1(G\star H)$ is
	$\circledast$-well-behaved w.r.t.\ $G$ and $H$ and $\MU_G=\MU_H=\emptyset$.
\end{proof}

We show now that maximum $1$-matchings $\fast(M_{G}, M_{H})$  of strong and lexicographic products
must always be perfect.
\begin{theorem}\label{thm:fast-1-matching}
	Let $\star\in\{\sprod,\lprod\}$. Let $G$ and $H$ be graphs
	with non-empty edge set and let $M_G\subseteq E_G$ and $M_H\subseteq E_H$. 
	The following statements are equivalent.
	\begin{enumerate}[itemsep=0.02em] 
	\item 	$\fast(M_{G}, M_{H})$ is a maximum $1$-matching of $G\star H$
					and thus, $\matchno_1(G\star H)$ is $\ast$-well-behaved
					w.r.t.\ $G$ and $H$.
	\item 	$M_G$ is a perfect $1$-matching of $G$
					and $M_H$ is a perfect $1$-matching of $H$.
	\item $\fast(M_{G}, M_{H})$ is a perfect $1$-matching of $G\star H$.
\end{enumerate}
\end{theorem}
\begin{proof}
  Assume that $\fast(M_{G}, M_{H})$ is a maximum $1$-matching. 
  Since both $G$ and $H$ have non-empty edge set, 
	$G\star H$ contains edges and thus,  $\fast(M_{G}, M_{H})\neq \emptyset$
	and $\matchno_1(G\star H) = |\fast(M_{G}, M_{H})|$.
	Prop.\ \ref{prop:matching-ast} implies that $M_G$
  is non-empty $1$-matching of $G$ and $M_H$ is 
  a non-empty $1$-matching of $H$. 
  In addition, Prop.\ \ref{prop:1-m-fast} implies that $\MU_G=\MU_H=\emptyset$.
  Since $M_G\neq \emptyset$ and $M_H\neq \emptyset$
  we have $\MU_G=\MU_H=\emptyset$ if and only if
  $u_G=u_H=0$ with 
   $u_G$ being the number of $M_G$-unmatched vertices in $G$
  and $u_H$ the number of $M_H$-unmatched vertices in $H$. 
  Hence, $M_G$ and $M_H$ must be perfect $1$-matchings. 
  Therefore, (1) implies (2).
  
  Now assume that $M_G$ is a perfect $1$-matching of $G$
	and $M_H$ a perfect $1$-matching of $H$. Thus $u_G=u_H=0$
	and 	Prop.\ \ref{prop:size-ast}
	implies that $|\fast(M_G,M_H)| = \frac{1}{2} |V_G||V_H| = \frac{1}{2} |V_{G\star H}|$, 
	which is, by Lemma \ref{lem:muFormula}, the size of a perfect 1-matching
	of $G\star H$.   Therefore, (2) implies (3).

	Finally, assume that $\fast(M_{G}, M_{H})$ is a perfect $1$-matching of $G\star H$.
	Thus, $\fast(M_{G}, M_{H})$ is, in particular,  
	a maximum $1$-matching of $G\star H$. 
	Since, in this case, $\matchno_1(G\star H) = |\fast(M_{G}, M_{H})|$, 
	we can conclude that $\matchno_1(G\star H)$ is $\ast$-well-behaved
	w.r.t.\ $G$ and $H$. Therefore, (3) implies (1).
\end{proof}

Similar results as in Thm.\ \ref{thm:fast-1-matching} for the direct product do not hold. 
By way of example, the graph $K_2\dprod P_3$ in Fig. \ref{fig:exmpl-dprod} (left) has 
a maximum $1$-matching $\fast(M_{K_2},M_{P_3})$ although only one factor has a perfect $1$-matching, 
Moreover, the existence of a  perfect $1$-matching in one of the factors is not sufficient
to conclude that $\fast(M_{G},M_{H})$ is a maximum $1$-matching in a direct product. 
To see this, consider the the graph $K_2\dprod K_3$ in Fig. \ref{fig:exmpl-dprod} (middle). 
In this example one factor has a perfect $1$-matching although $\fast(M_{K_2},M_{K_4})$ is not 
a maximum $1$-matching of $K_2\dprod K_3$.

For the construction of maximum $1$-matchings in strong and lexicographic 
products it does not matter whether we use  $\fcast(M_G,M_H)$ or $\fbast(M_G,M_H)$
as long as one of them is a maximum $1$-matching in the product.
\begin{proposition}\label{prop:maxfacbst}
	Let $\star\in\{\sprod,\lprod\}$. Let $G$ and $H$ be graphs
	with non-empty edge set and let $M_G\subseteq E_G$ and $M_H\subseteq E_H$. 
	The following statements are equivalent.
	\begin{enumerate}[itemsep=0.02em] 
	\item  $\fcast(M_G,M_H)$ is maximum $1$-matching of $G\star H$.
	\item $\fbast(M_G,M_H)$ is maximum $1$-matching of $G\star H$. 
	\item $\fbast(M_H,M_G)$ is maximum $1$-matching of $G\star H$. 
	\end{enumerate}
\end{proposition}
\begin{proof}
		Assume that $\fcast(M_G,M_H)$ is maximum $1$-matching of $G\star H$.
		Prop.\ \ref{prop:matching-cast}(M3) implies that $M_G$ and $M_H$ are
		1-matchings of $G$ and $H$, respectively. 
		By Cor.\ \ref{cor:matching-fac}, $\fbast(M_G,M_H)$ 
		and  $\fbast(M_H,M_G)$ are $1$-matchings of	 $G\star H$. 
		By Thm.\ \ref{thm:max-Wk-fbast} and \ref{thm:max-Wk-fcast}, 
		$\fcast(M_G,M_H)$, $\fbast(M_G,M_H)$ and
		$\fbast(M_H,M_G)$ are contained in $\mathcal{W}_1(G\star H, M_G,M_H)$
		and are, in particular, maximum-sized elements of $\mathcal{W}_1(G\star H, M_G,M_H)$.
		Consequently, $|\fcast(M_G,M_H)| = |\fbast(M_G,M_H)|=|\fbast(M_H,M_G)|$
		and thus, $\fbast(M_G,M_H)$ and $\fbast(M_H,M_G)$ are maximum $1$-matchings of $G\star H$. 
		Hence, (1) implies (2) and (3). 
		
		The equivalence between (2) and (3) is easy to verify by utilizing 
		Prop.\ \ref{prop:matching-fac} and Remark \ref{rem:perfect-fbast}
		together with Prop.\ \ref{prop:cardi}. 
		
		We finally show that (2) implies (1). Thus, assume that 
		$\fbast(M_G,M_H)$ is maximum $1$-matching of $G\star H$. 
		By Prop.\ \ref{prop:matching-fac} and Remark \ref{rem:perfect-fbast},
		it holds that $M_G$ is a perfect 1-matching of $G$ and $M_H=\emptyset$
		or $M_G$ and $M_H$ are 1-matchings of $G$ and $H$, respectively. In both cases,  
		Prop.\ \ref{prop:matching-cast}(M3) is satisfied and, therefore, 
		$\fcast(M_G,M_H)$ is a $1$-matching of $G\star H$.
		As argued above, $\fcast(M_G,M_H)$ and $\fbast(M_G,M_H)$ are maximum-sized elements 
		in $\mathcal{W}_1(G\star H, M_G,M_H)$ and thus,
		$|\fcast(M_G,M_H)| = |\fbast(M_G,M_H)|$. Hence, 
		$\fcast(M_G,M_H)$ is a maximum $1$-matchings of $G\star H$. 
\end{proof}

For general integer $k>1$, Prop.\ \ref{prop:maxfacbst} is not always satisfied. 
By way of example, consider non-empty perfect $k$-matchings  $M_G$ and $M_H$ of $G$ and $H$, respectively. 
In this case, $\fbast(M_G,M_H)$ is a $k$-matching of $G\star H$ while 
 $\fcast(M_G,M_H)$ is $k^2$-matching of $G\star H$.
This, in particular, implies that $\matchno_{k^2}(G\star H)$ could be 
$\circledast$-well-behaved  w.r.t.\ $G$ and $H$ although it is not 
				$\boxast$-well-behaved  w.r.t.\ $G$ and $H$ as $G$ and $H$
				may have only empty $k^2$-matchings.

\begin{corollary}\label{cor:fast-implies-fcbast}
	Let $\star\in\{\sprod,\lprod\}$. Let $G$ and $H$ be graphs
	with non-empty edge set and let $M_G\subseteq E_G$ and $M_H\subseteq E_H$. 
	If $\fast(M_G,M_H)$ is a maximum $1$-matching of $G\star H$, then 
	$\fcast(M_G,M_H)$ and  $\fbast(M_G,M_H)$ are perfect $1$-matchings of $G\star H$.
\end{corollary}
\begin{proof}
Assume that $\fast(M_G,M_H)$ is a maximum $1$-matching of $G\star H$. 
By Thm.\ \ref{thm:fast-1-matching}, this is, if and only if 
$M_G$ and $M_H$ are perfect 1-matchings of $G$ and $H$, respectively.
Hence, Prop.\ \ref{prop:matching-cast}(M3) is satisfied and
Prop.\ \ref{prop:size-cast} implies that 
$|\fcast(M_G,M_H)| = \frac{1}{2}n_Gn_H = \frac{1}{2}n_{G\star H}$, 
the size of a perfect 1-matching of $G\star H$. 
By Prop.\ \ref{prop:maxfacbst}, $\fcast(M_G,M_H)$ and $\fbast(M_G,M_H)$
must have the same size, which completes the proof.
\end{proof}

The converse of Cor.\ \ref{cor:fast-implies-fcbast} is not satisfied
in general. To see this, consider the graph $K_2\sprod K_3\simeq K_2\lprod K_3$
in  Fig. \ref{fig:exmpl-dprod} (right). Here,  $\fcast(M_G,M_H)$ is a 
perfect $1$-matching of $G\star H$ although $\fast(M_G,M_H)$ is not.

\section{Summary}
\label{sec:sum}

In this contribution, we considered the problem of finding $k$-matchings or,
equivalently, $k$-regular subgraphs in graph products. Since this problem is NP-hard
in general, we focused on possible constructions $\fbast(M_G,M_H)$, $\fcast(M_G,M_H)$
and $\fast(M_G,M_H)$ for such $k$-matchings in graph products $G\star H$ based on
subsets $M_G$ and $M_H$ of the edges of the respective factors $G$ and $H$. We
characterized under which conditions these constructions allow us to determine
maximum $k$-matchings in the products which leads to well-behaved $k$-matchings. All
these constructions were weak-homomorphism preserving and we showed that our
constructions have always maximum size among all weak-homomorphism preserving
$k$-matchings for specific products $\star$, which are comprised in the set
$\mathcal{W}_k(G\star H, M_G,M_H)$. 

These results offer many open interesting questions for further research. It would be
of interest to understand the graph classes whose products always have well-behaved $k$-matchings. 
Moreover, there are plenty of other weak-homomorphism preserving
$k$-matchings $\fM(M_G, M_H)$ that can be constructed in a reasonable way
from matchings in the factors, e.g., constructions based on $k_G$-matchings and
$k_H$-matchings that yield $k_G+k_H$-matchings in the products. Note,
$\fbast(M_G,M_H), \fcast(M_G,M_H), \fast(M_G,M_H)\notin \mathcal{W}_k(G\star H,
M_G,M_H)$ may be possible, in particular, if none of $\fbast(M_G,M_H),
\fcast(M_G,M_H)$ and $\fast(M_G,M_H)$ yield a $k$-matching. Hence, it is of interest
to provide different constructions $\fM(M_G, M_H)$ that cover these cases.
In particular, can one characterize
(near-perfect) weak-homomorphism preserving $k$-matchings in general? When are are
weak-homomorphism preserving $k$-matchings also maximum $k$-matchings? Which type of factors
yield products for which all (none, exactly one) 
maximum $k$-matchings are  also weak-homomorphism preserving?
  
Furthermore, the factorization of the considered graphs into two factors is, in
general, not unique. Thus, the question arises as whether it is possible to
characterize graphs $G$ for which the $k$-matchings are well-behaved w.r.t. every
(resp. some or none) factorization of $G$ into two or more non-trivial factors. 

Note the problem of finding a $k$-regular subgraph in graph products
   is NP-hard in general. Hence it might be of interest to study further constructions
   and to design efficient heuristics or even exact algorithms to solve this 
   intractable problem.

\bibliographystyle{spbasic}
\bibliography{biblio}

\end{document}